\documentclass[10pt]{article}
\usepackage{amsmath}
\usepackage{amsfonts}
\usepackage{amssymb}
\usepackage{mathrsfs}
\usepackage{amsthm}
\usepackage{graphicx}

\newtheorem{theorem}{Theorem}
\newtheorem{lemma}[theorem]{Lemma}
\newtheorem{proposition}{Proposition}
\newtheorem{corollary}[theorem]{Corollary}

\theoremstyle{definition}
\newtheorem{remark}{Remark}
\newtheorem{example}{Example}

\newcommand{\Diff}{\mathcal{D}}
\newcommand{\Diffmu}{\Diff_{\mu}}
\newcommand{\id}{\mathrm{id}}
\DeclareMathOperator{\sgrad}{sgrad}
\newcommand{\Laplacian}{\Delta}
\newcommand{\grad}{\nabla}
\DeclareMathOperator{\diver}{div}
\newcommand{\normal}{\nu}
\DeclareMathOperator{\Jac}{Jac}
\DeclareMathOperator{\curl}{curl}
\DeclareMathOperator{\cn}{cn}
\DeclareMathOperator{\sn}{sn}
\DeclareMathOperator{\dn}{dn}
\DeclareMathOperator{\zn}{zn}
\newcommand{\ellipticE}{\mathbf{E}}
\newcommand{\ellipticK}{\mathbf{K}}

\newcommand{\greekfcn}{\zeta}

\title{Conjugate point criteria on the area-preserving diffeomorphism group}
\author{Stephen C. Preston}

\begin{document}

\maketitle

\begin{abstract}
This paper answers some questions about conjugate points along the geodesics corresponding to steady 2D Euler flows, posed by a paper of Drivas-Misio{\l}ek-Shi-Yoneda. We present a new sufficient criterion for the existence of conjugate points, which improves on the criterion of Misio{\l}ek. It applies in any rotational cell of a steady 2D Euler flow, and in case of rotational symmetry it captures all known conjugate points. We give a general construction of the surfaces that admit a steady fluid with given area form, velocity profile, and vorticity profile, and from this we show how to detect conjugate points in a single rotational cell of a steady flow. When the velocity profile has a local extremum, the criterion becomes particularly simple. Several examples are provided, and in an appendix we use the Misio{\l}ek criterion to give some new examples of conjugate points along Kolmogorov flows on the torus.
\end{abstract}

\tableofcontents

\section{Introduction}\label{introsection}

Suppose $M$ is a two-dimensional compact surface, possibly curved and possibly with boundary, with a Riemannian metric
inducing an area form $\mu$ and a Poisson bracket $\{ \cdot, \cdot\}$. If $f\colon M\to \mathbb{R}$ satisfies the PDE
\begin{equation}\label{steadyeuler}
\{f, \Laplacian f\}= 0,
\end{equation}
then $f$ generates
a steady solution of the 2D Euler equation for ideal fluids, via $U = \sgrad f$. The velocity field $U$ in turn generates
a geodesic $\eta$ defined by
\begin{equation}\label{flowdef}
\frac{\partial \eta}{\partial t}(t,p) = U(\eta(t,p)), \qquad \eta(0,p)=p,
\end{equation}
and $\eta(t)$ a curve in the group of area-preserving diffeomorphisms $\Diffmu(M)$ with $\eta(0) = \id$.
In this paper we are interested in the question of whether there are conjugate points along this geodesic, i.e.,
whether $\eta(T)$ is conjugate to $\eta(0)$ for some $T>0$. This is closely related to the question of whether the
geodesic is minimizing between its endpoints, and is also related to the question of the stability of particle
trajectories.

Misio{\l}ek~\cite{Mis1993} was the first to find conjugate points in $\Diffmu(M)$ for the case $M=S^2$, along a rigid rotation (along with conjugate points on the $3$-ball and the $3$-sphere), taking advantage of the fact that $S^2$, $S^3$, and the spheres making up the $3$-ball all have conjugate points along them. A more difficult problem is to find conjugate points along a fluid geodesic when $M$ itself does not have conjugate points; this was solved soon after by Misio{\l}ek as well~\cite{misiolekconjugate}, who found them along the geodesic which comes from stream function $f(x,y) = \cos{6x} \cos{2y}$ on the torus $\mathbb{T}^2$. The method he used in this paper generalizes to give a technique for finding conjugate points along any steady flow, and is based on the well-known criterion in Riemannian geometry~\cite{docarmo}: if for some $T>0$ there is a vector field $\mathbf{Y}\colon [0,T]\to \Diffmu(M)$ with $\mathbf{Y}(0)=\mathbf{Y}(T)=0$ such that the index form
\begin{equation}\label{generalindexform}
I(\mathbf{Y},\mathbf{Y}) = \int_0^T \left\langle \frac{D\mathbf{Y}}{dt}, \frac{D\mathbf{Y}}{dt}\right\rangle + \langle \mathbf{R}(\mathbf{Y}, \dot{\eta})\dot{\eta}, \mathbf{Y}\rangle \, dt
\end{equation}
is negative, then there is a Jacobi field $J(t)$ vanishing at both $t=0$ and $t=\tau$ for some $\tau<T$, and hence $\eta(\tau)$ is conjugate to $\eta(0)$.

By considering variation fields of the form $\mathbf{Y}(t) = Y(t)\circ\eta(t)$ with $Y(t) = \sin{(\tfrac{t}{T})} W$, where $W$ is a time-independent vector in $T_{\id}\Diffmu(M)$ (i.e., a divergence-free vector field on $M$), Misio{\l}ek showed that negativity of the index form $I(\mathbf{Y},\mathbf{Y})$ for sufficiently large $T$ is implied by positivity of the following quantity:
\begin{equation}\label{misiolekcurvature}
MC(W,W) = \int_M \langle \nabla_{[U,W]}U + \nabla_U[U,W], U\rangle \, d\mu.
\end{equation}
Given a particular steady solution $U$ of the 2D Euler equation (for example $U=\sgrad f$ as above), finding a single divergence-free $W$ which makes $MC(W,W)>0$ is sufficient to prove that there are eventually conjugate points along the geodesic $\eta$ which is the flow of $U$. This is far easier than actually solving the Jacobi equation to find the conjugate points explicitly. The quantity $MC(W,W)$ is called ``Misio{\l}ek curvature'' by Tauchi and Yoneda (see \cite{tauchiyonedacoriolis}--\cite{tauchiyonedaarnold}) and the condition $MC(W,W)>0$ called the M-criterion (see \cite{tauchiyonedaellipsoid}). This criterion has been successfully applied to find conjugate points along many steady flows, and even along nonsteady Rossby-Haurwitz waves by Benn~\cite{benn}.

Nonetheless it is quite far from being a complete characterization. One drawback is that the test fields $Y$ are not necessarily related in any obvious way to the given steady field $U$; in fact the easiest way to apply the M-criterion is to guess the form of $Y$ based on some parameters and solve for them numerically. A bigger drawback is that although the criterion applies for certain rotationally symmetric surfaces like spheres with a bulge around the equator, it does not apply to rigid rotation on the standard sphere, which is the only case where we know all conjugate points and Jacobi fields explicitly. (See Example \ref{sphericaljacobis} below for details.) We can compute that $MC(W,W)\le 0$ for every divergence-free field $W$, yet we can show by direct construction that the rotation $\eta(\tau)$ is conjugate to the identity for $\tau = \frac{n(n+1)\pi}{k}$ whenever $k,n\in\mathbb{N}$ with $k\le n$. Essentially the difference is that instead of choosing $Y(t) = \sin{(\tfrac{t}{T})} \, W$, we choose $Y(t) = \sin{(\tfrac{t}{T})}\, \eta(\alpha t)_*W$, where we push-forward the velocity field by the Lagrangian flow for some constant $\alpha$.

It turns out that incorporating this additional drift term gives a much stronger condition, which captures many more conjugate points along two-dimensional steady flows. Moreover we obtain a simpler criterion which only depends on a single function. The idea is that we can work things out fairly well in the rotationally symmetric case, and the Jacobi field that generates the earliest conjugate point will have a $\theta$-dependence of the form $\cos{\theta}$; higher frequencies generate conjugate points happening later in time (but the frequency needs to be nonzero to generate any conjugate points). Even those steady flows which are not rotationally symmetric still generically have closed trajectories, and split $M$ into cells where they either rotate clockwise or counterclockwise. We will construct a variation field supported in one of these cells, with a low-frequency oscillation term and a time-dependent drift.

Now let us make this precise. We suppose $U$ is a steady solution of the Euler equation with $U=\sgrad f$ for a stream function $f$. Suppose $f$ has a nondegenerate local minimum at a point $p_0\in M$ (i.e., its Hessian is positive definite), and that $f$ has no other critical points and remains negative in the interior of a cell $C$, while $f=0$ on the boundary $\partial C$. (A stream function is only determined up to a constant, so this is no loss of generality.) The basic examples we have in mind are a function $f$ that depends solely on the radius on a rotationally symmetric manifold, and the Kolmogorov flows generated by Laplacian eigenfunctions $f(x,y) = -\cos{mx} \cos{ny}$ on the $2$-torus.
Within the cell $C$, the flow $\eta$ of $U$ consists of counterclockwise rotation along level sets of $f$ with finite period depending only on the value of $f$.

\begin{figure}[!ht]
\centering
\includegraphics[scale=0.35]{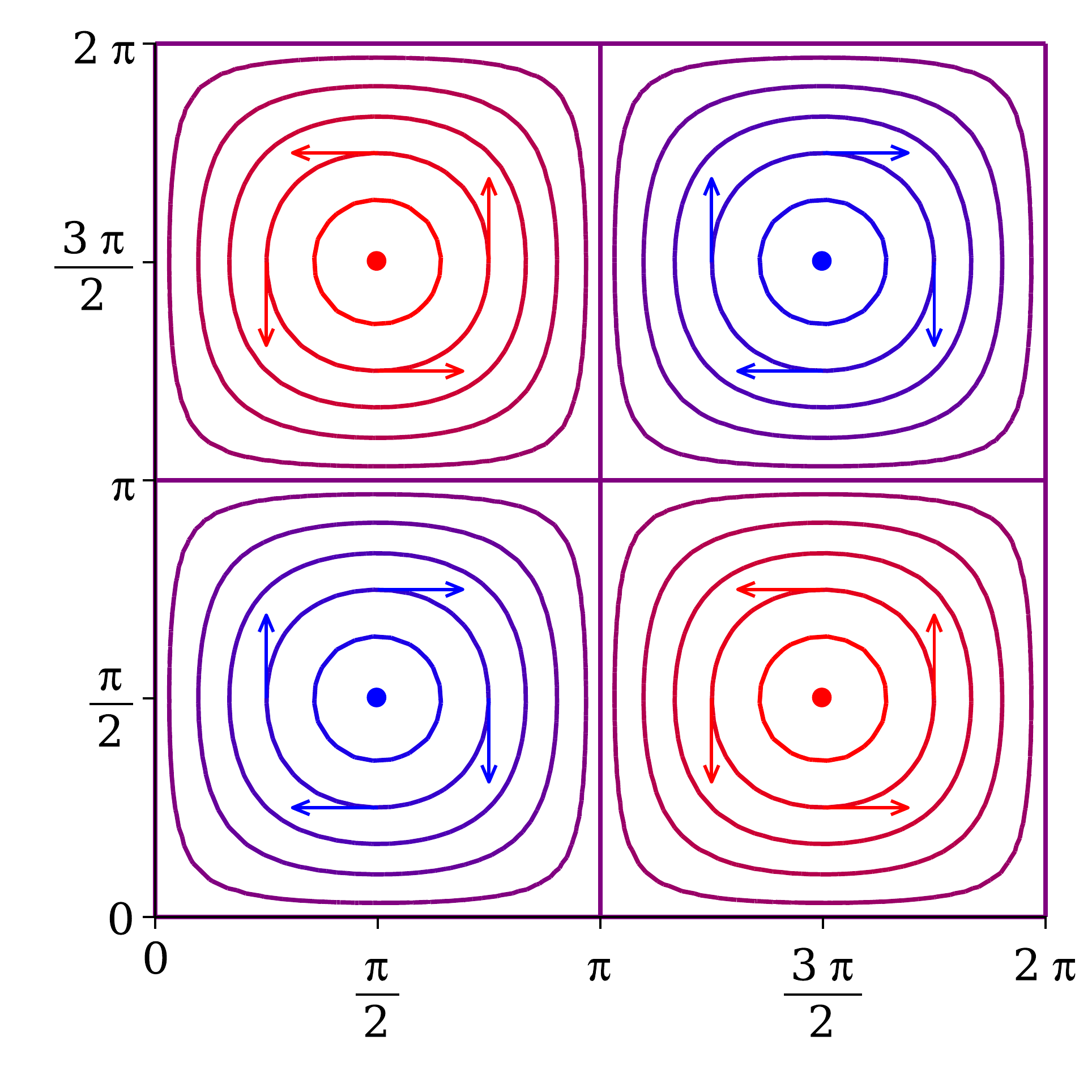}
\caption{Level sets of the function $f(x,y) = \sin{x}\sin{y}$, the simplest Kolmogorov flow on the torus. Red curves are negative values of $f$, while blue curves are positive values. The fluid rotates in each of the four cells independently, while the corner points are all hyperbolic fixed points.}
\end{figure}


We will show in Lemma \ref{polarcoordsgeneral} that in this situation, we can define a polar coordinate system $(r,\theta)$ with $r\in (0,R)$ and $\theta\in [0,2\pi)$ such that the stream function is given by $f = F(r)$, the steady velocity field is given by $U = u(r) \, \partial_{\theta}$, and the area form is given by $\mu = \varphi(r)\, dr\wedge d\theta$, for some functions $F$, $u$, and $\varphi$ satisfying $u(r) = F'(r)/\varphi(r)$. Here the metric components are given by
$$ds^2 = g_{11}(r,\theta) \, dr^2 + 2g_{12}(r,\theta) \, dr d\theta + g_{22}(r,\theta) \, d\theta^2,$$
with no simplifications except that
$$g_{11}g_{22}-g_{12}^2 = \varphi(r)^2 \qquad \text{and}\qquad
\frac{\partial}{\partial \theta}\big( u(r) g_{12}(r,\theta)\big) = \frac{\partial}{\partial r}\big( u(r) g_{22}(r,\theta)\big) - \varphi(r) \omega(r).$$
In the case where the speed $\lvert U\rvert$ is constant along trajectories, the manifold must be rotationally symmetric already, and these polar coordinates will reduce to the standard ones.

We prove in Lemma \ref{polarcoordsconverse} that the functions $F(r)$, $\varphi(r)$, and $G(r)$ may be specified arbitrarily (subject to the right behavior at the origin $r=0$ and the requirements that $F'$, $G$, and $\varphi$ are all positive for $r>0$) together with a single arbitrary but sufficiently small function $\greekfcn(r,\theta)$, and any such choice will generate a steady flow in a cell on some curved surface. When the function $\greekfcn$ is identically zero, we obtain rotationally symmetric examples. This principle can be used to make up more interesting examples, especially since it can be quite difficult to compute these quantities if the function and manifold are already given.

The main theorem of this paper is the following criterion for conjugate points along the geodesic $\eta$.

\begin{theorem}\label{mainindexthm}
Suppose we are given a solution $f$ of the steady Euler equation \eqref{steadyeuler}, with a unique critical point in a cell $C$ which is a local minimum, with $f=0$ on $\partial C$. Construct a polar coordinate chart $(r,\theta)$ as in Lemma \ref{polarcoordsgeneral}.
Define $E, G\colon [0,R]\to [0,\infty)$ by the formulas
\begin{equation}\label{EGdef}
E(r) = \frac{1}{2\pi} \int_0^{2\pi} g_{11}(r,\theta)\,d\theta, \qquad G(r) = \frac{1}{2\pi} \int_0^{2\pi} g_{22}(r,\theta)\,d\theta.
\end{equation}
Write $\Delta f = \omega(r)$ for the vorticity.

If for some $\alpha\in \mathbb{R}$ and some
$\xi\colon [0,R]\to \mathbb{R}$ such that $\xi(0)=\xi(R)=0$,
we have
\begin{equation}\label{alphacondition}
I = \int_0^R \frac{G(r)}{\varphi(r)} \left( \frac{d}{dr}\Big[ \big(\alpha - u(r))\xi(r)\Big]\right)^2 + \frac{E(r)}{\varphi(r)} \big( \alpha - u(r)\big)^2 \xi(r)^2 - \omega'(r)\big( \alpha - u(r)\big) \xi(r)^2 \, dr < 0,
\end{equation}
then $\eta(\tau)$ is conjugate to $\eta(0)=\id$ for some $\tau>0$.
\end{theorem}

The case where $f$ has a local maximum is identical and leads to a clockwise flow; we have chosen this case only because it gives the more standard polar coordinates.

Having chosen $\xi$, the quantity $I$ in \eqref{alphacondition} is a quadratic function $I = A \alpha^2 + 2B \alpha + C$ of $\alpha$, with $A>0$, which can be made negative for some $\alpha$ if and only if $B^2 - AC>0$.

If we assume that $u'(r)\ne 0$ in the cell (i.e., the velocity is both positive and monotone), we obtain the following alternative forms of Theorem \ref{mainindexthm}. The condition \eqref{indexVversion} is sometimes easier than \eqref{alphacondition}, while the condition \eqref{indexQversion} is typically harder than either, but the latter is useful to compare to previous results on the sign of the curvature of $\Diffmu(M)$, as well as for concluding that the index is always positive in some situations (see Section \ref{kolmosection}).

\begin{corollary}\label{indexthmQversion}
Suppose we have the same situation as in Theorem \ref{mainindexthm}. Define $v\colon [0,R]\to \mathbb{R}$ by
\begin{equation}\label{vdef}
v(r) = \frac{G'(r)u(r)}{2\varphi(r)}.
\end{equation}
Then  the index $I$ appearing in \eqref{alphacondition} can be written in the form
\begin{equation}\label{indexVversion}
I = \int_0^R \Big[ \frac{\big(\alpha - u(r)\big)^2}{\varphi(r)} \big( G(r) \xi'(r)^2 + E(r) \xi(r)^2\big) - 2 v'(r)\big(\alpha - u(r)\big) \xi(r)^2 \Big] \, dr.
\end{equation}
If in addition $u'(r)$ is nowhere zero on $(0,R)$, define
\begin{equation}\label{Qdef}
Q(r) = \frac{v'(r)}{u'(r)}.
\end{equation}
Then the index \eqref{indexVversion} may also be written in the form
\begin{equation}\label{indexQversion}
I = \int_0^R \big(\alpha - u(r)\big)^2 \left[ \frac{G(r)}{\varphi(r)} \Big( \xi'(r) - \frac{Q(r)\varphi(r)}{G(r)} \, \xi(r)\Big)^2 \\
- \frac{\varphi(r)}{G(r)} \, M(r) \xi(r)^2 \right] \, dr,
\end{equation}
where
\begin{equation}\label{Mdefinition}
M(r) := \frac{G(r)}{\varphi(r)} \, Q'(r) + Q(r)^2 - \frac{E(r)G(r)}{\varphi(r)^2}.
\end{equation}
\end{corollary}

The condition on the function $M$ defined by \eqref{Mdefinition} that
\begin{equation}\label{Qconditioncurvature}
M(r) > 0 \text{  for some $r\in (0,R)$}
\end{equation}
is sufficient to have some positive sectional curvature of $\Diffmu(M)$ in a section containing $U$; see \cite{nonpositive}.
Existence of such positive curvature 
is necessary but not sufficient for the existence of conjugate points; Corollary \ref{indexthmQversion} shows that one needs \eqref{Qconditioncurvature} to hold not merely at one point but on a sufficiently large interval $(a,b)\subset (0,R)$.

If the manifold is rotationally symmetric, i.e., there is a polar coordinate system with $ds^2 = dr^2 + \varphi(r)^2 \, d\theta^2$ for some function $\varphi(r)$ already, then any purely radial function $f = F(r)$ will satisfy \eqref{steadyeuler}, and the polar coordinates we construct in Lemma \ref{polarcoordsgeneral} will reduce to those. In that case we will have $E(r) = 1$ and $G(r) = \varphi(r)^2$. Furthermore in that case we do not need to assume $f$ has only one critical point in the cell $C$, since the only purpose of that assumption is to construct the $\theta$ coordinate manually; instead we can work on the entire manifold $M$. This situation occurs if and only if the speed $\lvert U\rvert$ is constant along trajectories, as we will see in Remark \ref{constantspeedremark}.

We now demonstrate several consequences of Theorem \ref{mainindexthm}. We begin with ``isochronal flows'' on a rotationally symmetric manifold, as defined by Drivas et al.~\cite{drivasmisiolek}, or in other words Killing fields of the rotationally symmetric metric.

\begin{corollary}\label{killingcorollary}
Suppose 
that the steady Euler velocity field is given by $U = \partial_{\theta}$, i.e., $u(r)\equiv 1$, so that the field is isochronal (i.e., all orbits have the same period $2\pi$). 

Then there is a conjugate point along $\eta$ if and only if $\frac{d}{dr}(G'(r)/\varphi(r))$ is not identically zero on $[0,R]$. In the rotationally symmetric case where $E(r)=1$ and $G(r)=\varphi(r)^2$, this reduces to the condition that
the Gaussian curvature is not identically zero on $M$.
\end{corollary}

In particular Corollary \ref{killingcorollary} captures the conjugate points we know for rotations on the $2$-sphere, as well as for any ellipsoid (see \cite{tauchiyonedaellipsoid}). The only rotationally symmetric situations in which we will not have a conjugate point are the flat torus with $\varphi(r)\equiv 1$ and rigid rotations in a disc with $\varphi(r)=r$. It is already known that the curvature operator is nonnegative in the first case by Misio{\l}ek~\cite{Mis1993} and in the second case by the author~\cite{nonpositive}, so we could not have gotten conjugate points anyway.

Above we have considered the cases where $u'$ is identically zero and where $u'$ is nowhere zero. It is also interesting to consider the case where $u'(r_0)$ is zero at one point (while assuming for nondegeneracy that $u''(r_0)\ne 0$). In \cite{nonpositive} we saw (in the rotationally symmetric case) that this was sufficient to find some positive sectional curvature in a plane containing $U$. To obtain a conjugate point one needs a little more.

\begin{corollary}\label{uprimecorollary}
Consider again the situation of Theorem \ref{mainindexthm}.
Suppose that the velocity function $u(r)$ has an isolated critical point at $r_0\in (0,R)$, with $u'(r_0)=0$ and $u''(r_0)\ne 0$. If
\begin{equation}\label{localextremecondition}
-\frac{u(r_0)}{\varphi(r_0)u''(r_0)} \frac{d}{dr}\Big|_{r=r_0}\left( \frac{G'(r)}{2\varphi(r)}\right) > \frac{9}{16},
\end{equation}
then there is eventually a conjugate point along the geodesic $\eta$.
In the rotationally symmetric case where $G(r)=\varphi(r)^2$ and $\varphi''(r) = -\kappa(r)\varphi(r)$ in terms of the Gaussian curvature $\kappa$, condition \eqref{localextremecondition} reduces to
$$ \frac{\kappa(r_0) u(r_0)}{u''(r_0)} > \frac{9}{16}.$$
\end{corollary}

As the reader might guess, the trick here is to choose $\alpha$ in the index \eqref{indexVversion} to be equal to the critical value $u(r_0)$, which makes $(\alpha - u(r)) \approx -\tfrac{1}{2} u''(r_0)(r-r_0)^2$. For $r$ close to $r_0$, the first two positive terms will thus look like $(r-r_0)^4$ while the last one looks like $(r-r_0)^2$, and thus it can be made to dominate if the test function $\xi$ is supported in a small neighborhood of $r_0$.

The velocity function $u(r)$ typically has a local maximum or minimum at $r=0$, so we can get a similar criterion to Corollary \ref{uprimecorollary} at $r=0$ by using the same trick. Things work out differently however since $E(r) = O(1)$, $\varphi(r)=O(r)$, and $G(r) = O(r^2)$ near $r=0$.

\begin{corollary}\label{uprimeorigincorollary}
Consider again the situation of Theorem \ref{mainindexthm}. Suppose $u(0)\ne 0$, $u'(0)=0$, and $u''(0)\ne 0$. If
\begin{equation}\label{originextremecondition}
3E(0) +12 G''(0) < \frac{2u(0)}{u''(0)} \left( \frac{\phi'''(0)}{\phi'(0)} - G^{iv}(0)\right).
\end{equation}
then there is eventually a conjugate point along the geodesic $\eta$.
In the rotationally symmetric case where $G(r) = \varphi(r)^2$ and $E(r)=1$, this reduces to
\begin{equation}\label{originextremerotational}
\frac{\kappa(0) u(0)}{u''(0)} > \frac{9}{4}.
\end{equation}
\end{corollary}

Finally we will use the results above to answer the three open questions posed by Drivas, Misio{\l}ek, Shi, and Yoneda in \cite{drivasmisiolek}.
\begin{enumerate}
\item For what values of $(m,n)$ does the Kolmogorov stream function $f(x,y) = -\cos{mx}\cos{ny}$ for $n$ and $m\ge n$ positive integers have a conjugate point along its geodesic? This had been established by Misio{\l}ek in \cite{misiolekconjugate} for the case $m=6$ and $n=2$, and extended in \cite{drivasmisiolek} to a condition equivalent to $n\ge 2$ and $m\ge \sqrt{3} (n +2/n)$. Theorem \ref{mainindexthm} does not produce conjugate points along Kolmogorov flows since we will show in Section \ref{kolmosection} that \eqref{indexQversion} is never satisfied. However we show that $(m,1)$ for $m\ge 2$, along with $(2,2)$, $(3,2)$, and $(3,3)$ all have conjugate points along their flows in the Appendix. We conjecture that all values of $(m,n)$ produce conjugate points except possibly $(1,1)$, but at this point we cannot prove the full result.
\item Can any Arnold stable solution on a surface $M$ have conjugate points? \textbf{Yes.} For simplicity take $u(r)=1$, so that $F'(r) = \varphi(r)$ and $\omega(r) = 2\varphi'(r)$. Then the ratio $\omega'(r)/F'(r)$ is given by $-2\kappa(r)$, in terms of the Gaussian curvature $\kappa$. On a disc with strictly negative curvature, the Arnold stability condition is satisfied, and Corollary \ref{killingcorollary} yields a conjugate point along the geodesic of rigid rotations. One could construct many other examples using the criterion of Theorem \ref{mainindexthm}. We give the details for a nonisochronal flow in Section \ref{examplesection}.
\item Are vortices with constant strength on elliptical domains free of conjugate points? \textbf{Yes.} In fact this is easy from the general index formula, which shows that if the vorticity is constant, then the third term in the index formula disappears (and what remains must be positive). See Proposition \ref{constantvorticityprop} below.
\end{enumerate}

\begin{theorem}\label{kolmocase}
Suppose $m$ and $n$ are both positive integers and $M=\mathbb{T}^2 = (\mathbb{R}/2\pi \mathbb{Z})^2$ with the usual flat metric $ds^2=dx^2+dy^2$. Let $f(x,y) = -\cos{mx}\cos{ny}$ be a Kolmogorov stream function, satisfying $\Delta f = -(m^2+n^2)f$ and thus \eqref{steadyeuler}. Consider the cell $C$ defined by $\lvert x\rvert \le \frac{\pi}{2m}$ and $\lvert y\rvert \le \frac{\pi}{2n}$, and define the radial coordinate $r$ by $r = \sqrt{1-f^2}$. Then for every nontrivial radial function $\xi$ with $\xi(0)=\xi(1)=0$, the index $I$ appearing in Theorem \ref{mainindexthm} is positive. Hence we cannot detect conjugate points by considering variations supported in a single cell.
\end{theorem}



Finally we discuss the three-dimensional case. Here it is much simpler to find conjugate points: instead of constructing a function $g(t,r,\theta)$ with $Y=\sgrad g$ a divergence-free vector field on $M$ and trying to show that $I(Y,Y)<0$, it is sufficient to find a vector field $Y(t,\eta(t,p_0))$ along a \emph{single} Lagrangian trajectory $t\mapsto \eta(t,p_0)$ such that the pointwise index form
$$ I_{p_0}(Y,Y) = \int_0^T \langle Z(t)_{\eta(t,p_0)}, Z(t)_{\eta(t,p_0)}\rangle + dU^{\flat}\big(Y(t)_{\eta(t,p_0)}, Z(t)_{\eta(t,p_0)}\big) \, dt < 0, \qquad Z(t) = \frac{DY}{dt} - \nabla_{Y(t)}U.$$
Given such a field along one trajectory, we can build a variation field supported in a small neighborhood of it which will make the full index form \eqref{generalindexform} negative on $\Diffmu(M)$. See \cite{hardest} for the criterion and some examples, along with \cite{wkb} for an improved criterion that captures the locations of these conjugate points. Examples in the axisymmetric case were given in \cite{washabaugh}. The fact that variations are easy to find in three dimensions is a reflection of the fact that the Riemannian exponential map is not Fredholm, which allows conjugate points to cluster along geodesics and have infinite order; see \cite{EMP}. The fact that we can easily find an infinite dimensional space of variations (all of which have disjoint supports) means that there are infinitely many conjugate points on a finite interval; the fact that the exponential map is Fredholm in two dimensions prevents this (see Misio{\l}ek~\cite{misiolekmorse}) and means that the energy-reducing variations must have overlapping support and cannot be selected in a small neighborhood of any point or trajectory. The only exceptions are Corollaries \ref{uprimecorollary} and \ref{uprimeorigincorollary}, where the supports of energy-reducing variations are in a small neighborhood of the fastest or slowest trajectory. 


Here is a summary of the paper's contents. In Section \ref{backgroundsection} we provide full details on conjugate points along 2D fluid flows, including the simplification of the index form in two dimensions along either steady or nonsteady flows. We also describe the explicit Jacobi fields and conjugate points along rigid rotations of the $2$-sphere, which inspires the main construction of the paper. In Section \ref{polarlemmasection} we describe the general construction of polar coordinates in a cell. We also prove Lemma \ref{polarcoordsconverse} which shows that any stream function $F(r)$, area form coefficient $\varphi(r)$, and vorticity function $\omega(r)$ of the radial coordinate can appear on some curved surface, and provides a parameterization of all such surfaces. In Section \ref{mainproofsection} we prove our main Theorem \ref{mainindexthm} by considering simple explicit variations and the long-time limit. We also prove Corollary \ref{indexthmQversion} by showing how to derive formulas \eqref{indexVorigin} and \eqref{indexQversion} from \eqref{alphacondition}. In Section \ref{corollaryproofsection} we prove Corollary \ref{killingcorollary} giving the necessary and sufficient condition for an isochronal flow to have a conjugate point. We also prove Corollaries \ref{uprimecorollary} and \ref{uprimeorigincorollary}, which are closely related and give the only local criteria we have along 2D flows. In Section \ref{kolmosection} we work out the polar coordinate representation for the Kolmogorov flows in terms of Jacobian elliptic functions, and prove Theorem \ref{kolmocase} which shows that the index form is positive. In Section \ref{examplesection} we work out the details of four examples of nontrivial flows where the new criteria of this paper are effective. Finally in the Appendix we use Misio{\l}ek's criterion from \cite{misiolekconjugate} to find additional Kolmogorov flows with conjugate points.

\section*{Acknowledgements:} The author conducted this reesarch while partially funded by Simons Foundation Collaboration Grant 318969. The author thanks Martin Bauer, Theo Drivas, and Gerard Misio{\l}ek for helpful discussions, as well as the organizers of the Workshop on Small Scale Dynamics in Fluid Motion at the Simons Center for Geometry and Physics, where the work was completed.

\section{Background}\label{backgroundsection}

We will give a brief introduction to the geometry of volumorphisms and the Euler equation as a geodesic on this group. For more details in the smooth case,
see Arnold-Khesin~\cite{AK1998}, along with the paper of Ebin-Marsden~\cite{EM1970} where the Sobolev manifold structure was discussed and Misio{\l}ek~\cite{Mis1993} where the basic properties of the Jacobi equation, conjugate points, and the index form were demonstrated. For what we need, we can assume everything is $C^{\infty}$, though the rigorous approach requires only that the diffeomorphisms and velocity fields are in $H^s$ for $s>\dim{M}/2+1$, and everything we do will make sense in that context as well.

Suppose $M$ is a Riemannian manifold, possibly with boundary $\partial M$. The Euler equation for a time -dependent vector field $U$, divergence-free and tangent to the boundary, is given by
\begin{equation}\label{eulerequation}
\frac{\partial U}{\partial t} + U\cdot \nabla U = -\grad p, \qquad U(0) = U_0.
\end{equation}
Here $p$ is a real-valued function determined up to a constant (so that $\grad p$ is uniquely determined) by
$$ \Laplacian p = -\diver{(U\cdot \nabla U)}, \qquad \langle \nabla p, \normal\rangle = -\langle U\cdot \nabla U, \normal\rangle.$$
We can write this more simply in terms of the projection $P$:
\begin{equation}\label{projectiondef}
P(X) = X + \grad f, \qquad \Laplacian f = -\diver{X}, \quad \langle \grad f, \normal\rangle = -\langle X, \normal\rangle.
\end{equation}
For any vector field $X$ on $M$, we see that $P(X)$ is divergence-free and tangent to $\partial M$. Hence equation \eqref{eulerequation}
becomes
\begin{equation}\label{eulerprojection}
\frac{\partial U}{\partial t} + P(\nabla_U U) = 0, \qquad U(0)=U_0.
\end{equation}

The flow of $U$ is a volume-preserving diffeomorphism $\gamma(t)$ given by
\begin{equation}\label{flowequation}
\frac{\partial \gamma}{\partial t}(t,x) = U\big(t,\gamma(t,x)\big), \qquad \gamma(0,x) = x,
\end{equation}
and in terms of $\gamma$ the Euler equation becomes the second-order equation
\begin{equation}\label{geodesiceuler}
\frac{\partial^2 \gamma}{\partial t^2} = -\grad p\circ\gamma, \qquad \gamma(0)=\id, \quad \dot{\gamma}(0) = U_0.
\end{equation}
The map $\eta$ satisfies, for each fixed time, the volume-preserving condition $\Jac(\eta)\equiv 1$,
and maps the boundary of $M$ to itself by the conditions on $U$. Thus it can be viewed as a curve in the
infinite-dimensional space $\Diffmu(M)$ of volume-preserving diffeomorphisms, or volumorphisms. Under composition,
this is a group. We now consider the manifold structure on it.

Elements of the tangent space $T_{\eta}\Diffmu(M)$ at a volumorphism $\eta\in \Diffmu(M)$ are given by
right-translations $\mathbf{Y} = Y\circ\eta$, where $Y\in T_{\id}\Diffmu(M)$ is a divergence-free vector field tangent to $\partial M$.
 The kinetic energy defines a Riemannian metric on $\Diffmu(M)$ by the formula
$$ \langle \mathbf{Y}, \mathbf{Y}\rangle_{\eta} = \int_M \langle Y\circ\eta, Y\circ\eta\rangle \, d\mu = \int_M \langle Y,Y\rangle \, d\mu,$$
the last formula representing right-invariance of the metric and following from the change of variables formula and $\Jac(\eta)\equiv 1$.
Along a curve $\eta(t)$, a time-dependent vector field $\mathbf{Y}(t)$ along $\eta(t)$ is similarly given by $\mathbf{Y}(t) = Y(t)\circ\eta(t)$
for each $t$. Right-invariance of the Riemannian metric leads to right-invariance of the covariant derivative and curvature.
The covariant derivative of $\mathbf{Y}$ along $\eta$ is given by
\begin{equation}\label{covariantderivative}
\frac{D\mathbf{Y}}{dt} = \left( \frac{\partial Y}{\partial t} + P(\nabla_{V(t)}Y(t))\right)\circ\eta(t),
\end{equation}
where $V$ is the right-translated tangent vector of $\eta$, i.e., $\frac{d\eta}{dt} = V\circ\eta$, and $\nabla_VY$ is the covariant derivative computed pointwise in $M$.

The Riemannian curvature tensor along the curve $\eta$ is given by
\begin{equation}\label{riemann}
\mathbf{R}(\mathbf{Y}, \dot{\eta})\dot{\eta} = \Big(P(\nabla_YP(\nabla_VV) - \nabla_VP(\nabla_YV) - \nabla_{[Y,V]}V)\Big)\circ\eta,
\end{equation}
again using the right-translations $\dot{\eta}=V\circ\eta$ and $\mathbf{Y}=Y\circ\eta$. Its importance comes from the equation in variations: if $\eta(t,\delta)$ is a family of geodesics depending on a parameter $\delta$, with $\eta(t,0)=\gamma(t)$, then the variation field
$$\mathbf{J}(t) = \frac{\partial \eta}{\partial \delta}\Big|_{\delta=0}$$
is called a Jacobi field, and satisfies the equation
\begin{equation}\label{jacobiequation}
\frac{D^2 \mathbf{J}}{dt^2} + \mathbf{R}(\mathbf{J}, \dot{\gamma})\dot{\gamma} = 0.
\end{equation}
This is the linearization of the geodesic equation.
If there is a time $T>0$ such that $\mathbf{J}(0)=0$ and $\mathbf{J}(T)=0$, then we say that there is a monoconjugate point along $\gamma$.
The length of the geodesic is infinitesimally minimizing up to the first monoconjugate point, and ceases to be minimizing after that. Jacobi
fields vanishing at two points represent stable perturbations of the particle paths, which return (to first order) to where they would have been
without the perturbation. By right-invariance, if $\mathbf{J}(t) = J(t)\circ\eta(t)$ for $J(t)\in T_{\id}\Diffmu(M)$, and $\dot{\gamma}(t)=U(t)\circ\eta(t)$, then the Jacobi equation \eqref{jacobiequation} may be rewritten as an equation on the Lie algebra:
\begin{equation}\label{jacobiequationtranslate}
\left(\frac{\partial}{\partial t} + \nabla_{U(t)}\right)^2 J(t) + \mathbf{R}\big(J(t), U(t)\big)U(t) = 0.
\end{equation}

The fact that the geodesic equation \eqref{geodesiceuler} splits into the two decoupled equations \eqref{eulerequation} is a consequence
of the right-invariance of the metric. It has the further consequence that the Jacobi equation \eqref{jacobiequation} splits into decoupled equations:
the system is given by
\begin{equation}\label{jacobisplit}
\frac{\partial J}{\partial t} + \nabla_UJ - \nabla_JU = K, \qquad \frac{\partial K}{\partial t} + P(\nabla_UK + \nabla_KU) = 0,
\end{equation}
where we view the first equation as defining $K$. The second equation is the linearized Euler equation, which is used to study perturbations
of the velocity field itself without looking at the perturbed particle paths. Solving the Jacobi equation in the form \eqref{jacobiequationtranslate} is rather complicated due to the fact that the curvature operator \eqref{riemann} is difficult to compute explicitly, due to the projections which
require solution of the nonlocal equation \eqref{projectiondef}. Even the decoupled version \eqref{jacobisplit} is difficult to find explicit solutions
for in general due to those same projections.
See \cite{prestonstability} for general discussion on this point and some simple cases. The following example is sketched in \cite{prestonthesis}, but we do not know if the full solution has ever appeared in print.

\begin{example}\label{sphericaljacobis}
On a simply-connected surface $M$, we may write $J = \sgrad{g}$ and $K = \sgrad{h}$, for some time-dependent functions $g$ and $h$ on $M$.
Equations \eqref{jacobisplit} then become the system
\begin{equation}\label{jacobisplit2d}
\frac{\partial g}{\partial t} + \{f,g\} = h, \qquad \frac{\partial \Delta h}{\partial t} + \{f, \Delta h\} + \{h, \Delta f\} = 0,
\end{equation}
where $U = \sgrad f$. See \cite{prestonstability} for details.

Now suppose $M$ is the round sphere, with spherical polar coordinates giving the metric  $ds^2 = dr^2 + \sin^2{r} d\theta$ and stream function $f = -\cos{r}$. Then the steady Euler velocity field is given by $U = \partial_{\theta}$, and the vorticity is given by $\Delta f = 2f$. The system \eqref{jacobisplit2d} simplifies in this case to the system
\begin{equation}\label{jacobi2sphere}
\frac{\partial g}{\partial t} + \frac{\partial g}{\partial \theta} = h, \qquad
\frac{\partial\Delta h}{\partial t} + \frac{\partial \Delta h}{\partial \theta} + 2 \, \frac{\partial h}{\partial \theta} = 0.
\end{equation}

If $P^k_n$ denotes the Legendre polynomial for positive integers $k$ and $n$, then the general solution of \eqref{jacobi2sphere}
with $g(0,r,\theta)=0$ is given by
\begin{multline*}
g(t,r,\theta) =  \sum_{n=1}^{\infty} c_n t P_n(\cos{r}) \\ +
\sum_{n=1}^{\infty} \sum_{k=1}^n \sin{\left(\frac{kt}{n(n+1)}\right)} \, P^k_n(\cos{r})
\left( a_{nk} \cos{k\left( \theta - t + \frac{t}{n(n+1)}\right)} + b_{nk}
\sin{k\left( \theta - t + \frac{t}{n(n+1)}\right)}\right),
\end{multline*}
for arbitrary coefficients $a_{nk}$, $b_{nk}$, $c_n$. This can be checked by plugging in, based on the fact that
$$\Delta \big(P^k_n(\cos{r}) \cos{k\theta}\big) = -n(n+1) P^k_n(\cos{r}) \cos{k\theta}.$$

Setting $c_n=0$ since those Jacobi fields grow linearly in time, we find that conjugate points happen at times $\tau = \frac{\pi n(n+1)}{k}$ for positive integers $k,n$ with $k\le n$. In addition the basic Jacobi field stream functions are all of the form
$$ g(t,r,\theta) = \sin{\left(\frac{t}{T}\right)} \xi(r) \cos{k(\theta-\alpha t)}$$
for $\alpha = 1-\frac{1}{n(n+1)}$. In particular $\alpha$ can \emph{never} be zero. See Figure \ref{jacobispherepics}.
\end{example}

\begin{figure}[!ht]
\centering
$t=0$ \hspace{1.2in} $t=\pi$  \hspace{1.2in} $t=2\pi$  \hspace{1.2in} $t=3\pi$ \\
\includegraphics[scale=0.45]{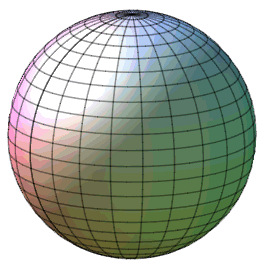} \quad 
\includegraphics[scale=0.45]{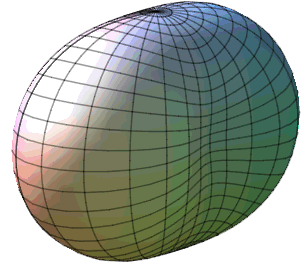} \quad
\includegraphics[scale=0.45]{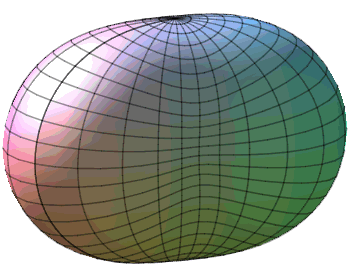} \quad
\includegraphics[scale=0.45]{spherejacobi-3.png}
\caption{The transported Jacobi field stream function $g(t,r,\theta+t)$, dragged by the uniform rotation, when $k=2$ and $n=2$, given by 
$g(t,r,\theta+t) = \sin{\frac{t}{3}} \sin^2{r} \sin{(2 \theta +\frac{t}{3})}$. As $t$ goes from $0$ to $\tau=3\pi$, we get a conjugate point as $g$ expands and shrinks, while rotating a quarter-turn around the sphere due to the drift.}\label{jacobispherepics}
\end{figure}

An easier way to detect conjugate points, without explicitly finding the Jacobi fields, is to use the index form. Multiplying the Jacobi equation \eqref{jacobiequation} by $\mathbf{J}(t)$ and integrating by parts from $t=0$ to $t=T$ gives the formula
$$ I(\mathbf{J},\mathbf{J}) = \int_0^T \left\langle \frac{D\mathbf{J}}{dt}, \frac{D\mathbf{J}}{dt}\right\rangle_{L^2(M)} - \langle \mathbf{R}(\mathbf{J},\dot{\gamma})\dot{\gamma}, \mathbf{J}\rangle_{L^2(M)} \, dt = 0$$
if $\mathbf{J}$ is a Jacobi field vanishing at $t=0$ and $t=T$. This index makes sense for any vector field along $\gamma$, not just a Jacobi field, and it can be shown that we have $I(\mathbf{Y},\mathbf{Y})=0$ for some $\mathbf{Y}$ satisfying $\mathbf{Y}(0)=\mathbf{Y}(T)=0$ if and only if $\mathbf{Y}$ is a Jacobi field. Furthermore if $I(\mathbf{Y},\mathbf{Y})<0$ for some such field $\mathbf{Y}$, then there is a Jacobi field $\mathbf{J}$ along $\gamma$ vanishing at $t=0$ and for $t=\tau$ for some $\tau<T$. Establishing negativity of the index form for some vector field $\mathbf{Y}$ is thus an effective way to show there must be a conjugate point without actually having to find it.

Since the metric, covariant derivative, and Riemann curvature tensor are all right-invariant, the index form is as well, and we can write (using the formulas \eqref{covariantderivative} and \eqref{riemann})
\begin{multline}\label{indexformrightinvariant}
I(Y,Y) = \int_0^T \int_M \left\langle\frac{\partial Y}{\partial t} + P(\nabla_UY), \frac{\partial Y}{\partial t} + P(\nabla_UY)\right\rangle
\\
- \langle P\big(\nabla_YP(\nabla_UU)\big) - P\big(\nabla_UP(\nabla_YU)\big) + P(\nabla_{[U,Y]}U), Y\rangle \, d\mu \, dt.
\end{multline}
In fact it turns out this formula simplifies quite a lot, and in fact the divergence-free projections $P$ do not need to be computed at all. The following version of the formula is similar to one that appears in the author's work \cite{hardest}. (Note that we do this slightly differently from Misio{\l}ek~\cite{misiolekconjugate} who first assumes that $U$ is a steady solution of the Euler equation.)

\begin{lemma}\label{indexsimplifiedlemma}
Suppose $U$ solves the Euler equation \eqref{eulerprojection} on $[0,T]$. For a time-dependent family of divergence-free vector fields $Y(t)$ on $M$, tangent to the boundary of $M$, and vanishing at $t=0$ and $t=T$, the index form \eqref{indexformrightinvariant} takes the form
\begin{equation}\label{indexformhardest}
I(Y,Y) = \int_0^T \int_M \langle Z(t), Z(t)\rangle + dU(t)^{\flat}\big(Y(t),Z(t)\big) \, d\mu \, dt, \qquad Z(t) := \dot{Y}(t) + [U(t),Y(t)].
\end{equation}
\end{lemma}

\begin{proof}
First express $\tfrac{\partial Y}{\partial t} = Z - [U,Y]$ where $Z$ is defined in terms of $Y$ by \eqref{indexformhardest}. Then the covariant derivative of $Y$ is
\begin{equation}\label{covarderivz}
\frac{\partial Y}{\partial t} + P(\nabla_UY) = Z + P(\nabla_UY) - [U,Y] = Z + P(\nabla_YU),
\end{equation}
using the fact that
$$ P(\nabla_UY) - P(\nabla_YU) = P(\nabla_UY-\nabla_YU) = P([U,Y]) = [U,Y],$$
since the Lie bracket of two vector fields in $T_{\id}\Diffmu(M)$ is also in it by the Lie group property.
Observe that $Z(t)$ must also be in $T_{\id}\Diffmu(M)$.

Using formula \eqref{covarderivz}, the index form becomes
\begin{multline}\label{indexstep1}
I(Y,Y) = \int_0^T \int_M \langle Z + P(\nabla_YU), Z + P(\nabla_YU) \rangle
- \langle P\big(\nabla_YP(\nabla_UU)\big), Y\rangle \\
+ \langle P\big(\nabla_UP(\nabla_YU)\big),Y\rangle - \langle P(\nabla_{[U,Y]}U), Y\rangle \, d\mu \, dt.
\end{multline}
To simplify further, we use the fact that if $V$, $W$, and $X$ are any fields in $T_{\id}\Diffmu(M)$, then we have
\begin{equation}\label{trivialprojection}
\int_M \langle P(\nabla_VW), X\rangle \, d\mu = \int_M \langle \nabla_VW, X\rangle \, d\mu,
\end{equation}
since $P$ is the $L^2$ orthogonal projection onto $T_{\id}\Diffmu(M)$ and $X$ is already in that subspace. Thus the index \eqref{indexstep1} becomes
\begin{multline}\label{indexstep2}
I(Y,Y) = \int_0^T \int_M \langle Z, Z\rangle + 2\langle \nabla_YU, Z\rangle + \langle P(\nabla_YU),\nabla_YU \rangle
- \langle \nabla_YP(\nabla_UU), Y\rangle \\
 + \langle \nabla_UP(\nabla_YU),Y\rangle - \langle \nabla_{[U,Y]}U, Y\rangle \, d\mu \, dt.
\end{multline}

We next use integration by parts, based on the fact that $\int_M \langle X,\grad f\rangle \, d\mu = 0$ for any function $f$ and any element $X\in T_{\id}\Diffmu(M)$. This implies that if $V, W,X\in T_{\id}\Diffmu(M)$, then
\begin{equation}\label{integrationbyparts}
\begin{split}
\int_M \langle \nabla_VW, X\rangle \, d\mu &= \int_M \langle V, \grad\langle W,X\rangle \rangle \, d\mu -\int_M \langle W, \nabla_VX\rangle \, d\mu\\
&= -\int_M \langle W, \nabla_VX\rangle \, d\mu.
\end{split}
\end{equation}
Using this in \eqref{indexstep2} gives
\begin{multline}\label{indexstep3}
I(Y,Y) = \int_0^T \int_M \langle Z, Z\rangle + 2\langle \nabla_YU, Z\rangle + \langle P(\nabla_YU),\nabla_YU \rangle \\
+ \langle P(\nabla_UU), \nabla_YY\rangle - \langle P(\nabla_YU),\nabla_UY\rangle - \langle \nabla_{[U,Y]}U, Y\rangle \, d\mu \, dt.
\end{multline}
Note that we can combine the two terms containing $P(\nabla_YU)$ to get
$$ \int_M \langle P(\nabla_YU),\nabla_YU \rangle - \langle P(\nabla_YU),\nabla_UY\rangle \, d\mu = -\int_M \langle P(\nabla_YU),[U,Y]\rangle \, d\mu = -\int_M \langle \nabla_YU, [U,Y]\rangle,$$
using \eqref{trivialprojection} again. In addition the fact that $U$ solves the Euler equation means that $P(\nabla_UU) = -\tfrac{\partial U}{\partial t}$.
With these two facts in mind, the index form \eqref{indexstep3} now becomes
\begin{equation}\label{indexstep4}
I(Y,Y) = \int_0^T \int_M \langle Z, Z\rangle + 2\langle \nabla_YU, Z\rangle - \langle \nabla_YU,[U,Y] \rangle
- \langle \tfrac{\partial U}{\partial t}, \nabla_YY\rangle - \langle \nabla_{[U,Y]}U, Y\rangle \, d\mu \, dt.
\end{equation}

Integrating by parts in time, using that $Y(0)=Y(T)=0$, gives for the $U_t$ term that
\begin{align*}
-\int_0^T \int_M \langle \tfrac{\partial U}{\partial t}, \nabla_YY\rangle\, d\mu \, dt
&= \int_0^T \int_M \langle U, \nabla_{\tfrac{\partial Y}{\partial t}}Y \rangle + \langle U, \nabla_Y \tfrac{\partial Y}{\partial t} \rangle \, d\mu \,dt \\
&= \int_0^T \int_M \langle U, \nabla_{Z-[U,Y]}Y\rangle + \langle U, \nabla_Y(Z-[U,Y]) \, d\mu\, dt\\
&= \int_0^T \int_M \langle \nabla_{[U,Y]}U, Y\rangle - \langle \nabla_Z U, Y\rangle + \langle \nabla_YU, [U,Y]\rangle - \langle \nabla_YU, Z\rangle \, d\mu \,dt.
\end{align*}
Inserting this into \eqref{indexstep4} gives, after some cancellations, that
$$ I(Y,Y) = \int_0^T \int_M \langle Z, Z\rangle + \langle \nabla_YU, Z\rangle - \langle \nabla_Z U, Y\rangle  \, d\mu \, dt.$$
This becomes \eqref{indexformhardest} upon realizing that Cartan's formula gives
\begin{equation}\label{cartanformula}
dU^{\flat}(Y,Z) = \langle \nabla_YU, Z\rangle - \langle \nabla_ZU, Y\rangle.
\end{equation}
\end{proof}

In the special case where $U$ is independent of time, Misio{\l}ek~\cite{misiolekconjugate} obtained a simple criterion for
existence of conjugate points, which we reproduce below.

\begin{theorem}[Misio{\l}ek]\label{misiolektheorem}
Suppose $U$ is a steady (time-independent) solution of the Euler equation \eqref{eulerequation}. If there is a $W\in T_{\id}\Diffmu(M)$ such that
\begin{equation}\label{misiolekcriterion}
MC(W,W) := \frac{\langle \nabla_{[U,W]}U + \nabla_U[U,W], W\rangle_{L^2(M)}}{\langle W,W\rangle_{L^2(M)}} > 0,
\end{equation}
then there is eventually a conjugate point along the geodesic $\gamma$ which is the flow of $U$, i.e., for some $\tau>0$ there is a Jacobi field
$\mathbf{J}(t)$ along $\gamma$ such that $\mathbf{J}(0)=\mathbf{J}(\tau)=0$.
\end{theorem}

\begin{proof}
For $T>0$ we take a test field of the form
\begin{equation}\label{misiolektest}
Y(t,x) = \sin{\left(\frac{t}{T}\right)} \, W(x), \qquad \text{where } W\in T_{\id}\Diffmu(M).
\end{equation}
Obviously $Y(0)=Y(\pi T)=0$.
Then $Z$ from \eqref{indexformhardest} becomes
$$Z(t) = \frac{1}{T} \cos{\left(\frac{t}{T}\right)} \, W + \sin{\left(\frac{t}{T}\right)} \, [U,W].$$
Performing the time integration first, we easily get
$$ \int_0^{\pi T} \langle Z,Z\rangle + dU^{\flat}(Y,Z) \, dt =
\frac{\pi T}{2} \Big( \langle [U,W], [U,W]\rangle + dU^{\flat}\big(W,[U,W]\big) + \frac{1}{T^2} \langle W,W\rangle\Big).$$
Thus we compute in the limit as $T\to\infty$ that
\begin{equation}\label{misioleklimit}
\lim_{T\to \infty} \frac{2}{\pi T} I(Y,Y) = \int_M \langle [U,W], [U,W]\rangle + dU^{\flat}\big(W, [U,W]\big) \, d\mu.
\end{equation}
If the right side is negative for some $W$, then for sufficiently large $T$ the index form $I(Y,Y)$ will be negative, and hence there will be a conjugate point location occurring at some $\tau<\pi T$.

The right side of \eqref{misioleklimit} can be written using the Cartan formula \eqref{cartanformula} again to get
\begin{align*}
\int_M \langle [U,W], [U,W]\rangle + dU^{\flat}\big(W, [U,W]\big) \, d\mu
&= \int_M  \langle [U,W], [U,W]\rangle + \langle \nabla_WU, [U,W]\rangle - \langle \nabla_{[U,W]}U, W\rangle \, d\mu \\
&= \int_M \langle \nabla_UW, [U,W]\rangle - \langle \nabla_{[U,W]}U,W\rangle \, d\mu \\
&= -\langle W,W\rangle_{L^2(M)} MC(W,W),
\end{align*}
using the integration-by-parts formula \eqref{integrationbyparts} in the last line. Thus negativity of \eqref{misioleklimit} is equivalent to positivity of $MC(W,W)$.
\end{proof}

The expression $MC(W,W)$ is called the \emph{Misio{\l}ek curvature} of $W$, and Theorem \ref{misiolektheorem} is called the M-criterion in \cite{tauchiyonedaellipsoid}--\cite{tauchiyonedaellipsoid}.

Now let us specify to the two-dimensional case, not necessarily assuming that $U$ is steady (though in our examples it will be). To make things simpler, we assume that either $M$ is simply-connected or that at least the support of $Y$ is in a simply-connected cell (which will happen later when $M=\mathbb{T}^2$). In this case a divergence-free vector field $Y$ can be expressed as the skew-gradient of a stream function $g$, defined in terms of the area form $\mu$ so that
$$ 
\mu(X, \sgrad g) = dg(X) \text{ for every vector field $X$ on $M$.}$$
Since $Y$ must be tangent to the boundary of $M$, we can show that $g$ must be constant on the boundary of $M$, and we can choose this constant to be zero. Lemma \ref{indexsimplifiedlemma} then becomes the following.

\begin{corollary}\label{indexformcoro2D}
Suppose $M$ is a simply-connected surface (or portion of a larger surface), and that $U=\sgrad f$ solves the Euler equation \eqref{eulerprojection} on $[0,T]$. For a time-dependent family of functions $g(t)$ on $M$, vanishing on the boundary of $M$ and at $t=0$ and $t=T$, the index form \eqref{indexformhardest} for $Y=\sgrad g$ becomes
\begin{equation}\label{indexformhardestsgrad}
I(Y,Y) = \int_0^T \int_M \langle \nabla h, \nabla h\rangle + \Delta f \{g,h\}  \, d\mu \, dt, \qquad h := \frac{\partial g}{\partial t} + \{f,g\}.
\end{equation}
\end{corollary}

\begin{proof}
See \cite{prestonstability} for details of these computations.
Expressing $U = \sgrad f$, the vorticity $2$-form becomes
$dU^{\flat} = \Delta f \, \mu$ in terms of the usual Laplacian on functions. With $Y=\sgrad g$ and $Z=\sgrad h$, the formula \eqref{indexformhardest} relating $Y$ and $Z$ can be written in terms of the stream functions as
\begin{equation}\label{streamfunctionYZ}
g_t + \{f,g\} = h
\end{equation}
where the Poisson bracket is defined in terms of the area form $\mu$ by the formula
$$ df\wedge dg = \{f,g\} \, \mu.$$
This follows from the fact that $[\sgrad g, \sgrad g] = \sgrad \{g,h\}$.

We then have $$dU^{\flat}(\sgrad g, \sgrad h) \, \mu = \Laplacian f \,\mu(\sgrad g, \sgrad h) \,\mu = (\Laplacian f) dg\wedge dh = (\Laplacian f) \{g,h\} \, \mu,$$ which gives the second term. The first term comes from the fact that if $\mu$ comes from a Riemannian metric, then $\langle \sgrad h,\sgrad h\rangle = \lvert \grad h\rvert^2$.
\end{proof}

Corollary \ref{indexformcoro2D} makes it easy to answer the third question posed by Drivas et al.~\cite{drivasmisiolek}: if the velocity field has constant vorticity, there can be no conjugate points. A special case of this is the rigid rotation of a disc by $U = \partial_{\theta}$, where the vorticity function is $\curl{U} = 2$, for which the absence of conjugate points (or indeed any positive curvature) is known~\cite{nonpositive}.

\begin{proposition}\label{constantvorticityprop}
If $M$ is simply connected and $U=\sgrad f$ where $\Laplacian f$ is constant on $M$ and $f$ is constant on $\partial M$, then $U$ is a steady solution of the Euler equation on $M$, and there are no conjugate points along the geodesic $\gamma$ which is the flow of $U$.
\end{proposition}

\begin{proof}
The curl of the Euler equation \eqref{eulerequation} is given by
$$ \frac{\partial \Laplacian f}{\partial t} + \{f, \Laplacian f\} = 0,$$
so that if $\Laplacian f$ is constant in both space and time, it solves this equation. Note that $M$ must have a boundary in order for any nonconstant function to have constant nonzero Laplacian, since on a compact manifold we have $\int_M \Laplacian f \, d\mu = 0$.

Writing $Y = \sgrad g$ and $Z=\sgrad h$ as in Corollary \ref{indexformcoro2D}, we must have that $g$ is constant on the boundary. Since $h = g_t + \{f,g\}$ and $f$ is also constant on the boundary, we know that $df$ and $dg$ are collinear, so that $df\wedge dg=0$ and thus $\{f,g\}=0$. So $h$ is also constant on the boundary.

The second term in \eqref{indexformhardestsgrad} is now given by
$$  \int_M \Delta f \{g,h\}  \, d\mu = \Delta f \int_M dg \wedge dh =
\Delta f \int_{\partial M} g dh,$$
and this is zero since $h$ is constant on $\partial M$ for each time $t$. Therefore only the positive-definite term remains in \eqref{indexformhardestsgrad}, so $I(Y,Y)>0$ for every nontrivial field $Y$, and there can be no conjugate points along $\gamma$.
\end{proof}

If $\omega = \Delta f$ is piecewise constant instead of constant on all of $M$, the situation changes. Consider for example solutions of the Euler equation with $\omega = \curl{U}$ in $L^{\infty}(M)$. These exist as weak solutions of the Euler equation by a famous result of Yudovich~\cite{yudovich}. These are not in $H^s$ for $s>2$, and thus the results of Ebin-Marsden~\cite{EM1970} and \cite{Mis1993} do not apply; in particular there may not be a geodesic which integrates the velocity field, and the Jacobi equation may not be a reasonable ODE along the corresponding weak geodesic. Nonetheless one might expect to be able to make sense of a vorticity function that is very close to piecewise-constant on patches in a limiting sense if the index form still makes sense there. For this purpose if we suppose that $\omega$ is given by a nonzero constant $\omega_i$ on each simply connected open set $O_i\subset M$ and zero otherwise, then the second term of the index form \eqref{indexformhardestsgrad} would become
$$ \int_M \Delta f \{g, h\} \, d\mu =
\sum_i \omega_i \int_{O_i} dg\wedge dh = -\sum_i \omega_i \int_{\partial O_i} h \,dg,$$
and it is conceivable that this can be made negative for some choice of $g$. For example one could consider a Kirchoff ellipse, where there is a single patch of nonzero vorticity in an ellipse, which must then rotate at a constant speed. (See for example Majda-Bertozzi~\cite{majdabertozzi} for details.) This would take us too far afield here, but would be very interesting to study.

\section{Polar coordinates in a steady fluid cell}\label{polarlemmasection}

We make the following assumptions: there is an open cell $C\subset M$ for which the closure $\overline{C}$ is compact and simply connected, and in which a steady solution $U$ of the Euler equation $P(\nabla_UU) = 0$ is expressed by $U=\sgrad f$. We assume that $f$ has a local minimum on $C$ at some unique point $p\in C$, that $\nabla^2f$ is nondegenerate at $p$, and that $f=0$ on $\partial C$, and that $f$ has no critical points in $\overline{C}$ except for $p$ and points in $\partial C$. (In particular $f$ is a Morse function on $C$ and is always negative.) This is no loss of generality: adding a constant to $f$ does not change the velocity field $U$, and flipping the sign of $f$ would merely reverse the direction of the flow. The two main examples we have in mind are rotationally invariant functions $f(r)$ on rotationally symmetric surfaces with metric $ds^2 = dr^2 + \varphi(r)^2 \, d\theta^2$, and the function $f(x,y) = -\cos{mx}\cos{ny}$ on the flat torus $\mathbb{T}^2$, where in the latter case the cell is $C = (-\tfrac{\pi}{2m}, \tfrac{\pi}{2m})\times (-\tfrac{\pi}{2n}, \tfrac{\pi}{2n})$.

Since there are no other critical points, all level sets of $f$ except the singleton $\{p\}$ and the boundary $\partial C$ are one-dimensional compact manifolds, and hence are circles. The velocity field $U$ is nowhere zero on these level sets, since $\lvert U\rvert = \lvert \grad f\rvert$ is nowhere zero. By our choice of sign of $f$, these flows are all counterclockwise. Let $\gamma_t$ be the flow of $U$ defined by \eqref{flowdef}. Choose a curve $c\colon [0,R]\to M$ such that $c(0)=p$ and $c'(s)\ne 0$ for all $s\in [0,R]$. Then

\begin{lemma}\label{polarcoordsgeneral}
Suppose $f\colon M\to\mathbb{R}$ satisfies the steady Euler equation $\{f, \Delta f\} = 0$, and let $U=\sgrad f$.
Suppose $f$ has a nondegenerate local minimum at $p_0\in M$,
and no other critical points in a simply connected region $C$ with compact closure $\overline{C}$, with $f=0$ on $\partial C$. Then the level sets of $f$ contained in $C$ are all diffeomorphic to circles, and there is a polar coordinate chart $(r,\theta)$ for $0\le r\le R$ and $\theta\in [0,2\pi)$ for $C$ +with $r=0$ at $p_0$, $r=R$ on $\partial C$ such that $f(p) = F(r)$ for some $F\colon [0,R]\to \mathbb{R}$, with the area form given by
\begin{equation}\label{areaformpolar}
\mu = \varphi(r) \, dr\wedge d\theta
\end{equation}
for some function $\varphi\colon [0,R]\to \mathbb{R}_+$
satisfying
\begin{equation}\label{varphilimits}
\varphi(0)=0 \quad \text{and}\quad \varphi'(0)\ne 0
\end{equation}
and the velocity field satisfying
\begin{equation}\label{velocitypolar}
U = \frac{F'(r)}{\varphi(r)} \, \frac{\partial}{\partial \theta}.
\end{equation}
\end{lemma}

\begin{proof}
On $C\backslash \{p_0\}$, we have that $df$ is nonzero, so the implicit function theorem says that the level set $f^{-1}\{k\}$ is a smooth one-dimensional closed submanifold of $M$, and so it must be diffeomorphic to a circle.
Construct a curve $c\colon [0,R]\to M$ such that $c(0)=p_0$ with $c'(0)\ne 0$, and chosen so that $c$ intersects all level sets of
$f$ orthogonally; we can do this by solving the differential equation $c'(s) = \beta(s) \frac{\grad f(c(s))}{\lvert \grad f(c(s))\rvert}$ for $s>0$, for any positive function $\beta$, since $\grad f$ is nowhere zero except at $p_0$.

For each $s\in (0,R)$, the vector field $U$ is tangent to the level set $\{p\in C\, \vert \, f(p)=f(c(s))\}$, and points counterclockwise along it. It is nowhere zero since $\grad f$ is nowhere zero on $C\backslash \{p_0\}$, and its flow $\gamma$ defined by
$$ \frac{\partial \gamma}{\partial t}(t,p) = U\big(\gamma(t,p)\big)$$
satisfies $f(\gamma(t,p)) = f(p)$ for all $t\in \mathbb{R}$. For each $s\in (0,R)$ there is a smallest positive number $P(s)$ such that $\gamma(P(s),c(s)) = c(s)$, which is the period. Because $U$ is steady, $\gamma$ is a one-parameter group, which implies $\gamma(t+P(s), c(s)) = \gamma(t,c(s))$ for all $t\in \mathbb{R}$.

Now define polar coordinates $\Phi\colon (0,R)\to [0,2\pi)\to M$
by $$\Phi(r,\theta) = \gamma\big(\frac{P(r)\theta}{2\pi}, c(r)\big).$$
Then $\Phi$ maps $(0,R)\times [0,2\pi)$ bijectively onto $C\backslash \{p_0\}$, and extends $2\pi$-periodically for all $\theta\in\mathbb{R}$, while sending the origin $r=0$ to the local minimum $p_0$. We have $f(\Phi(r,\theta)) = f(c(r))$, so define $F\colon [0,R)\to \mathbb{R}$ by $F(r) = f(c(r))$. Since the flow of $U$ satisfies
$$ \gamma\big(t,\Phi(r,\theta)\big) = \gamma\Big(t, \gamma\big(\frac{P(r)\theta}{2\pi}, c(r)\big)\Big)
= \gamma\Big( t + \frac{P(r)\theta}{2\pi}, c(r)\Big) = \Phi\Big(r, P(r)\big(\theta+\frac{2\pi t}{P(r)}\big)\Big),$$
we see that the flow is given in polar coordinates by $(t,r,\theta) \mapsto (r,\theta+\tfrac{2\pi t}{P(r)})$. Thus we conclude that the velocity field $U$ must be given in coordinates by
\begin{equation}\label{Upolargeneral}
U = \frac{2\pi}{P(r)} \, \frac{\partial}{\partial \theta}.
\end{equation}

The area form in $(r,\theta)$ coordinates is given by
$$ \mu = \varphi(r,\theta) \, dr\wedge d\theta$$
for some function $\varphi$, and the divergence of $U$ in these coordinates is given by
$$ \diver{U} = \frac{1}{\varphi(r,\theta)} \, \frac{\partial}{\partial \theta} \Big( \frac{2\pi \varphi(r,\theta)}{P(r)}\Big) = \frac{2\pi P(r)}{\varphi(r,\theta)} \, \frac{\partial \varphi(r,\theta)}{\partial \theta}.$$
But since we know $U$ is divergence-free, we conclude that $\varphi$ depends only on $r$, which is \eqref{areaformpolar}.

Recall that $U$ is the skew-gradient of $f$, which is given in $(r,\theta)$ coordinates by $f(\Phi(r,\theta)) = F(r)$. The skew-gradient depends only on
the area form (not the metric) and is given by \eqref{velocitypolar} for a purely radial function with radial area form.

The local behavior at $r=0$ comes from working in different coordinates that make sense there. Since the metric is locally Euclidean and $f$ has a nondegenerate critical point at $p_0$, we can choose coordinates $(x,y)$ such that $ds^2 = dx^2 + dy^2 + O(x^2+y^2)$, rotated if necessary so that
$$f(x,y) = f(p_0) + \frac{1}{2} a x^2 + \frac{1}{2} by^2 + O((x^2+y^2)^{3/2}$$ for some $a>0$ and $b>0$. The area form in these coordinates is $\mu = dx\wedge dy + O(x^2+y^2)$, and the skew-gradient of $f$ is given by
$$ U = -\frac{\partial f}{\partial y} \, \frac{\partial}{\partial x} + \frac{\partial f}{\partial x} \, \frac{\partial}{\partial y} = -by \, \partial_x + ax \, \partial_y + O(x^2+y^2),$$
so that the flow is given near $p_0$ by
$$ \gamma(t,x,y) = \Big( x \cos{(\sqrt{ab}t)} -\sqrt{\tfrac{b}{a}} y \sin{(\sqrt{ab}t)},  \sqrt{\tfrac{a}{b}} x \sin{(\sqrt{ab}t)}
y \cos{(\sqrt{ab}t)} \Big) + O(x^2+y^2).$$

In particular the limit of the periods defined above is $\displaystyle \lim_{r\to 0} P(r) = \frac{2\pi}{\sqrt{ab}}$.
In addition on the overlap of the coordinate charts, if $(x,y) = \big(c_1(s),c_2(s)\big)$ defines the radial curve $c(s)$ in locally Euclidean coordinates, then the coordinates are related by
\begin{align*}
(x,y) &= \Phi(r,\theta) =
\gamma\Big(\frac{P(r)\theta}{2\pi}, c(r)\Big)
= \Big( c_1(r) \cos{(\sqrt{ab}P(r)\theta)/2\pi} -\sqrt{\tfrac{b}{a}} c_2(r) \sin{(\sqrt{ab}P(r)\theta/2\pi)}, \\
&\qquad\qquad  \sqrt{\tfrac{a}{b}} c_1(r) \sin{(\sqrt{ab}P(r)\theta/2\pi)}
c_2(t) \cos{(\sqrt{ab}P(r)\theta/2\pi)} \Big) + O(r^2) \\
&= \Big( rc_1'(0) \cos{\theta} - \sqrt{\tfrac{b}{a}} r c_2'(0) \sin{\theta}, \sqrt{\tfrac{a}{b}} r c_1'(0) \sin{\theta} + r c_2'(0) \cos{\theta}\Big) + O(r^2).
\end{align*}
Since by assumption $c'(0)\ne 0$, this gives a standard polar coordinate transformation near the origin, and so all the functions involved behave exactly as they would in standard polar coordinates, for example \eqref{varphilimits}.
\end{proof}

With $\omega = \Laplacian f$, the vorticity satisfies $\{f,\omega\}=0$ in the simply connected cell $C$. Since $f$ is strictly increasing from $p_0$ to $\partial C$, we can express the vorticity as a function of $f$, via $\Laplacian f = \Omega\circ f$. In our polar coordinates we can simply write $\omega = \omega(r)$. We now see how to compute this in terms of the metric components.

\begin{corollary}\label{vorticitycoords}
In the coordinates $(r,\theta)$ given by Lemma \ref{polarcoordsgeneral},
suppose the metric is given by
$$ ds^2 = g_{11}(r,\theta) \, dr^2 + 2g_{12}(r,\theta) \, dr\,d\theta +
g_{22}(r,\theta) \, d\theta^2.$$
Then the vorticity $\omega = \Laplacian f$ is given by the equation
\begin{equation}\label{vorticityGdef}
\omega(r) = \frac{1}{\varphi(r)} \, \frac{d}{dr} \big( G(r) u(r)\big),
\qquad \text{where } G(r) = \frac{1}{2\pi} \int_0^{2\pi} g_{22}(r,\theta) \, d\theta.
\end{equation}
Equivalently we may write
\begin{equation}\label{Gshortcut}
G(r) = \frac{1}{u(r)} \int_0^r \varphi(s) \omega(s)\,ds.
\end{equation}

In addition the components satisfy
\begin{equation}\label{componentsdeterminant}
g_{11}(r,\theta) g_{22}(r,\theta) - g_{12}(r,\theta)^2 = \varphi(r)^2
\end{equation}
and
\begin{equation}\label{crosscomponentderivative}
\frac{\partial g_{12}(r,\theta)}{\partial \theta} = \frac{1}{u(r)} \, \frac{\partial}{\partial r} \Big( u(r) \big[ g_{22}(r,\theta) - G(r)\big]\Big).
\end{equation}
\end{corollary}

\begin{proof}
Recall that the velocity field is given by $U=u(r) \, \frac{\partial}{\partial \theta}$. The vorticity function $\omega$ is given by
$ dU^{\flat} = \omega \, \mu,$
and so the first thing we do is compute the lift of $U$ to a $1$-form:
$$ U^{\flat} = u(r) g_{12}(r,\theta) \, dr + u(r) g_{22}(r,\theta) \, d\theta.$$
Then
\begin{equation}\label{generalvorticity}
dU^{\flat} = \left( \frac{\partial}{\partial r} \big( u(r)g_{22}(r,\theta)\big) - \frac{\partial}{\partial \theta}\big( u(r)g_{12}(r,\theta)\big)\right) \, dr\wedge d\theta.
\end{equation}
We conclude that
\begin{equation}\label{fullvorticity}
\frac{\partial}{\partial r} \big( u(r)g_{22}(r,\theta)\big) - u(r) \, \frac{\partial g_{12}(r,\theta)}{\partial \theta} = \omega(r) \varphi(r).
\end{equation}
Now the metric components $g_{12}$ and $g_{22}$ are $2\pi$-periodic functions of $\theta$, and we may integrate around this formula around the circle to obtain \eqref{vorticityGdef}.

Note that subtracting the averaged formula \eqref{vorticityGdef} from the full vorticity formula \eqref{fullvorticity} gives the formula \eqref{crosscomponentderivative}. Finally the formula \eqref{componentsdeterminant} follows from the usual formula for the Riemannian area form in terms of the metric components, and the formula \eqref{areaformpolar} that we already know. The formula \eqref{Gshortcut} comes from solving \eqref{vorticityGdef} for $G$ and using the fact that $G(0)=0$.
\end{proof}

\begin{remark}\label{constantspeedremark}
Observe that
$ \lvert U\rvert^2 = u(r)^2 g_{22}(r,\theta)$.
Thus $\lvert U\rvert^2$ is constant along trajectories if and only if $g_{22}(r,\theta)$ is independent of $\theta$. If that is the case, then $g_{22}(r,\theta) = G(r)$ as in \eqref{vorticityGdef}, and this means that $g_{12}(r,\theta)$ is constant. It is therefore determined by its values when $\theta=0$, i.e., along the curve $c$. If we choose $c$ to be parallel to the gradient of $f$ (i.e., perpendicular to $U$), then $g_{12}(r,\theta)\equiv 0$, and the metric becomes diagonal. Furthermore since $g_{22}(r,\theta) = G(r)$ and $g_{11}(r,\theta) G(r) = \varphi(r)$, we also conclude that $g_{11}(r,\theta) = E(r)$, so the metric is $ds^2 = E(r) \,dr^2 + G(r) \, d\theta^2$. We may rescale $r$ so that $E\equiv 1$ (corresponding to choosing $c$ to be unit speed), which means $G(r) = \varphi(r)^2$, and obtain the standard polar coordinate metric
$$ ds^2 = dr^2 + \varphi(r)^2 \, d\theta^2.$$
We have thus shown the standard result that $\lvert U\rvert$ is constant on trajectories if and only if $U$ is rotationally symmetric on a rotationally symmetric manifold.

If $\lvert U\rvert$ is not constant on trajectories, then there is no particular reason to demand that $c$ be orthogonal to level sets or that it be unit speed, since those properties will not be preserved as $c$ is rotated around the cell $C$ by the flow.
\end{remark}

Now we consider how much flexibility there is to construct steady fluid flows with given functions $F(r)$, $\varphi(r)$, $E(r)$, and $G(r)$. (Note that these functions also determine $u(r)$ and $\omega(r)$ by formulas \eqref{velocitypolar} and \eqref{vorticityGdef}.

\begin{lemma}\label{polarcoordsconverse}
Suppose $F(r)$, $\varphi(r)/r$, and $G(r)$ are given smooth functions of $r\in [0,R)$ that extend to smooth even functions on $(-R,R)$. (In particular $\varphi(r)$ is a smooth odd function of $r$.) Assume that $F'(r)$, $\varphi(r)$, and $G(r)$ are all positive for $r\in (0,R)$, with $G(0)=0$. Suppose $\greekfcn(r,\theta)$ is a smooth function on the disc $r\in [0,R]$, $\theta\in S^1$, i.e., it can be expressed as
$$\greekfcn(r,\theta) = \sum_{n=0}^{\infty} r^n a_n(r) \cos{n\theta} + \sum_{n=1}^{\infty} r^n b_n(r) \sin{n\theta},$$
where each $a_n$ and $b_n$ extends to a smooth even function of $r$.

Then there is a topological disc $M$ with Riemannian metric
$$ ds^2 = g_{11}(r,\theta) \, dr^2 + 2g_{12}(r,\theta) \, dr\,d\theta + g_{22}(r,\theta) \, d\theta^2,$$
where
\begin{equation}\label{generalmetriccomponents}
\begin{split}
g_{11}(r,\theta) &= \frac{\varphi(r)^2}{G(r)} \, \frac{1 + \frac{1}{F'(r)^2} \, \big(\frac{\partial \greekfcn}{\partial r}(r,\theta)^2\big)}{1 + \frac{\varphi(r)}{F'(r) G(r)} \, \frac{\partial \greekfcn}{\partial \theta}(r,\theta)} \\
g_{12}(r,\theta) &= \frac{\varphi(r)}{F'(r)} \, \frac{\partial \greekfcn}{\partial r}(r,\theta) \\
g_{22}(r,\theta) &= G(r) +  \frac{\varphi(r)}{F'(r)} \, \frac{\partial \greekfcn}{\partial \theta}(r,\theta),
\end{split}
\end{equation}
such that the area form is $dA = \varphi(r) \, dr \wedge d\theta$,
and such that the velocity field $U = \frac{F'(r)}{\varphi(r)} \, \partial_{\theta}$ is a steady solution of the 2D Euler equation on $M$
with vorticity given by \eqref{vorticityGdef}.

Conversely all discs with steady flows with stream function strictly increasing from a single minimum must arise in this way.
\end{lemma}

\begin{proof}
It is clear that
\begin{equation}\label{metricdeterminant}
g_{11}g_{22}-g_{12}^2 = \varphi(r)^2\end{equation}
by the definitions.
Thus the area form is $dA = \varphi(r) \, dr \wedge d\theta$, and the divergence of $U = u(r) \,\partial_{\theta}$ is indeed zero. In particular if $U$ is the skew-gradient of $F$, with $u(r) = F'(r)/\varphi(r)$.

It remains to show that the curl of $U$ is a function of $r$ alone. We have already computed in \eqref{generalvorticity} that the curl of $U=u(r)\,\partial_{\theta}$ in a general metric is a function given by
$$\curl{U} = \frac{1}{\varphi(r)} \left( \frac{\partial}{\partial r} \big( u(r)g_{22}(r,\theta)\big) - \frac{\partial}{\partial \theta}\big( u(r)g_{12}(r,\theta)\big)\right),$$
which in this case reduces to
\begin{align*}
\curl{U} &= \frac{1}{\varphi(r)} \left( \frac{\partial}{\partial r} \Big( \frac{F'(r)G(r)}{\varphi(r)} + \frac{\partial \greekfcn}{\partial \theta} \Big) - \frac{\partial}{\partial \theta}\Big( \frac{\partial \greekfcn}{\partial r}\Big)\right) \\
&= \frac{1}{\varphi(r)} \, \frac{d}{dr} \left( \frac{F'(r)G(r)}{\varphi(r)}\right).
\end{align*}
We conclude that $U\cdot \nabla \curl{U}=0$, so indeed $U$ is a steady solution of the Euler equation.

The conditions at $r=0$ are required in order to get smoothness of the metric, since $\varphi(r) \, dr$ must be a smooth $1$-form on $M$, and $g_{22}(0,\theta)$ must be zero. We then get smooth metric coefficients since $u(r) = F'(r)/\varphi(r)$ is a smooth, nowhere-zero function on $[0,R)$.

The converse simply comes from \eqref{fullvorticity},
after writing the metric as $g_{22}(r,\theta) = G(r) + \psi(r,\theta)$, where $\psi(r,\theta)$ has mean zero for each fixed $r$. Then we have that
$$\frac{\partial}{\partial r} \big( u(r)\psi(r,\theta)\big) = u(r) \, \frac{\partial g_{12}(r,\theta)}{\partial \theta}.$$
We conclude from the Poincar\'e Lemma since the disc is simply connected that there must be a smooth function $\greekfcn$ such that $u(r)\psi(r,\theta) = \frac{\partial}{\partial \theta}\greekfcn(r,\theta)$ and $u(r) g_{12}(r,\theta) = \frac{\partial}{\partial r}\greekfcn(r,\theta)$. Finally the formula for $g_{11}$ comes from solving the equation \eqref{componentsdeterminant} for $g_{11}$ in terms of $\varphi$ and the other components.
\end{proof}

\section{Proofs of Theorem \ref{mainindexthm} and Corollary \ref{indexthmQversion}}\label{mainproofsection}

We consider a variation field supported on the cell $C$ as in the previous section, of the form $$Y(t,r,\theta) = \sin{(\tfrac{t}{T})} \sgrad \big( \xi(r) \cos{(\theta-\alpha t)}\big),$$
on the time interval $[0,\pi T]$.
Performing the time integral first in the index formula \eqref{indexformhardestsgrad} turns out to make things much easier due to the following Lemma.

\begin{lemma}\label{timeintegrallemma}
Suppose $\alpha\in \mathbb{R}$ is nonzero.
For any $\theta\in \mathbb{R}$ and any vectors $p, q, r, s\in \mathbb{R}^n$ we have the following limits:
\begin{align}
\lim_{T\to\infty} \frac{4}{\pi T} \int_0^{\pi T}
\sin^2{\big(\frac{t}{T}\big)} \, dt &= 2 \label{sin2limit} \\
\lim_{T\to\infty} \frac{4}{\pi T} \int_0^{\pi T}
\sin^2{\big(\frac{t}{T}\big)} \cos^2{(\theta-\alpha t)} \, dt &= 1
\label{sin2cos2limit} \\
\lim_{T\to\infty} \frac{4}{\pi T} \int_0^{\pi T}
\Big( \sin{\big(\frac{t}{T}\big)} \sin{(\theta-\alpha t)} \, p +
\frac{1}{T} \cos{\big(\frac{t}{T}\big)} \cos{(\theta-\alpha t)}\, q\Big)^2 \, dt &=  \lvert p\rvert^2, \label{psinqcoslimit} \\
\lim_{T\to\infty} \frac{4}{\pi T} \int_0^{\pi T}
\Big( \sin{\big(\frac{t}{T}\big)} \cos{(\theta-\alpha t)} \, p -
\frac{1}{T} \cos{\big(\frac{t}{T}\big) \sin{(\theta-\alpha t)}\, q}\Big)^2 \, dt &= \lvert p\rvert^2, \label{pcosqsinlimit}
\end{align}
\begin{equation}\label{dotproductlimit}
\begin{split}
&\lim_{T\to\infty} \frac{4}{\pi T} \int_0^T
\Big\langle  \sin{\big(\frac{t}{T}\big)} \sin{(\theta-\alpha t)} p +
 \frac{1}{T} \cos{\big(\frac{t}{T}\big)} \cos{(\theta-\alpha t)} q, \\
 &\qquad\qquad\qquad
 \sin{\big(\frac{t}{T}\big)} \cos{(\theta-\alpha t)} r  - \frac{1}{T} \, \cos{\big(\tfrac{t}{T}\big)} \sin{(\theta-\alpha t)} s \Big\rangle \, dt =0.
\end{split}
\end{equation}
\end{lemma}

\begin{proof}
These are easy to evaluate from explicit formulas. For \eqref{sin2limit} we have
$$ \frac{4}{\pi T} \int_0^{\pi T} \sin^2{\big(\frac{t}{T}\big)} \, dt = 2$$
even without taking the limit.
For \eqref{sin2cos2limit} with $\alpha\ne 0$, we get
$$\frac{4}{\pi T} \int_0^{\pi T}
\sin^2{\big(\frac{t}{T}\big)} \cos^2{(\theta-\alpha t)} \, dt
= 1 - \frac{\sin{(\pi T\alpha)} \cos{(2\theta-\pi T\alpha)}}{\pi T \alpha (T^2\alpha^2-1)},
$$
and the limit of this is obviously $1$ as $T\to\infty$.


For the next two limits, note that the terms involving $\langle p, q\rangle$ and $\lvert q\rvert^2$ have terms $\frac{1}{T}$ and $\frac{1}{T^2}$ respectively multiplying them, so they are of obviously smaller order and disappear in the limit; thus \eqref{psinqcoslimit} follows from \eqref{sin2cos2limit}, while \eqref{pcosqsinlimit} follows since it's the same as \eqref{psinqcoslimit} with $\theta$ shifted by $\pi/2$.

For the last limit, by the same reasoning (ignoring the terms that have $\frac{1}{T}$), the limit is equal to
\begin{align*}
\langle p, r\rangle \lim_{T\to\infty} \frac{4}{\pi T} \int_0^{\pi T} \sin^2{\big(\frac{t}{T}\big)} \sin{(\theta-\alpha t)} \cos{(\theta-\alpha t)} \, dt
&= -\langle p,r\rangle \lim_{T\to\infty} \frac{\sin{(\pi T\alpha)} \sin{(2\theta-\pi T\alpha)}}{8\alpha(\alpha^2T^2-1)} = 0.
\end{align*}
\end{proof}

Note that if $\alpha=0$, instead of \eqref{sin2cos2limit} we get
$$ \frac{4}{\pi T} \int_0^{\pi T}
\sin^2{\big(\frac{t}{T}\big)} \cos^2{\theta} \, dt = 2\cos^2{\theta} = 1 + \cos{2\theta},$$
so it really does matter that $\alpha\ne 0$ in the Lemma above. If $\alpha$ is close to zero, we may need to choose $T$ very large in order to actually get close to this limit.
Similarly if $\alpha=0$ then \eqref{dotproductlimit} becomes $\sin{2\theta}\langle p,r\rangle$ rather than zero. We will eventually be integrating with respect to $\theta$, so the resulting terms probably integrate to zero anyway, but it is less obvious.

\begin{proof}[Proof of Theorem \ref{mainindexthm}]
For $T>0$ and $\alpha\in \mathbb{R}$, with $\xi\colon [0,R]\to \mathbb{R}$ satisfying the conditions $\xi(0)=0$, $\xi(R)=0$, and $\xi'(0)\ne 0$, let
\begin{equation}\label{gformula}
g(t,r,\theta) = \sin{(\tfrac{t}{T})} \xi(r) \cos{(\theta-\alpha t)}.
\end{equation}
Note that the conditions on $\xi$ at $r=0$ are necessary and sufficient to make $g$ continuous and differentiable at $p_0$, since near $p_0$ we will have $$g(t,r,\theta) \approx \sin{(\tfrac{t}{T})}\xi'(0) r \cos{(\theta-\alpha t)} \approx \sin{(\tfrac{t}{T})} \xi'(0) (x \cos{\alpha t} + y \sin{\alpha t})$$
in terms of the locally Euclidean coordinates discussed in Lemma \ref{polarcoordsgeneral}.

We compute that $h$ defined by \eqref{streamfunctionYZ} is given by
\begin{equation}\label{hformula}
\begin{split}
h(t,r,\theta) &= \frac{\partial g}{\partial t}(t,r,\theta) + u(r) \, \frac{\partial g}{\partial \theta}(t,r,\theta) \\
&= \frac{1}{T} \, \cot{(\tfrac{t}{T})} g(t,r,\theta) +
\big(\alpha - u(r)\big) \sin{(\tfrac{t}{T})} \xi(r) \sin{(\theta-\alpha t)}.
\end{split}
\end{equation}

In particular we get
\begin{equation}\label{ghbracket}
\begin{split}
\{g,h\} &= \frac{1}{\varphi} \, \left( \frac{\partial g}{\partial r} \, \frac{\partial h}{\partial \theta} - \frac{\partial g}{\partial \theta}\,\frac{\partial h}{\partial r}\right) \\
&= \frac{\sin^2{(\tfrac{t}{T})} \xi(r)}{\varphi(r)} \left( \xi(r) u'(r) \cos^2{(\theta-\alpha t)} + \frac{d}{dr} \big( (\alpha - u(r))\xi(r)\big) \right).
\end{split}
\end{equation}

We also compute from \eqref{hformula} that
\begin{multline}\label{gradh}
\grad h =
\left( \frac{1}{T} \, \cos{\big(\frac{t}{T}\big)} \cos{(\theta-\alpha t)} \xi'(r) + \frac{d}{dr}\Big(\big(\alpha - u(r)\big)\xi(r)\Big) \sin{\big(\frac{t}{T}\big)} \sin{(\theta-\alpha t)} \right) \grad r \\
+ \left( \big(\alpha - u(r)\big) \sin{\big(\frac{t}{T}\big)} \xi(r) \cos{(\theta-\alpha t)}  - \frac{1}{T} \, \cos{\big(\tfrac{t}{T}\big)} \xi(r) \sin{(\theta-\alpha t)} \right) \grad\theta.
\end{multline}

Referring to \eqref{indexformhardestsgrad}, we compute (without integrating in space) using formulas \eqref{psinqcoslimit} and \eqref{pcosqsinlimit} from Lemma \ref{timeintegrallemma} that
\begin{equation}\label{limitgradhsquared}
\begin{split}
&\lim_{T\to\infty} \frac{4}{\pi T} \int_0^T \lvert \grad h(t,r,\theta)\rvert^2 + \omega(r) \{g(t,r,\theta),h(t,r,\theta)\} \, dt \\
&\qquad\qquad =
\left\lvert \frac{d}{dr}\Big(\big(\alpha - u(r)\big)\xi(r)\Big) \right\rvert^2  \lvert \grad r\rvert^2 +
\left\lvert \big(\alpha - u(r)\big)  \xi(r)\right\rvert^2 \, \lvert \grad\theta\rvert^2.
\end{split}
\end{equation}
The term that would involve $\langle \grad r,\grad \theta\rangle$ has zero limit by formula \eqref{dotproductlimit}.

Similarly using \eqref{sin2limit} and \eqref{sin2cos2limit}, we see from equation \eqref{ghbracket} that
\begin{equation}\label{limitpoisson}
\begin{split}
\lim_{T\to\infty} \frac{4}{\pi T} \omega(r) \{g(t,r,\theta),h(t,r,\theta)\}
&=  \frac{\xi(r)}{\varphi(r)} \left( \xi(r) u'(r) + 2 \frac{d}{dr} \big( (\alpha - u(r))\xi(r)\big) \right) \\
&= \frac{1}{\varphi(r)} \frac{d}{dr} \Big( \big(\alpha-u(r)\big) \xi(r)^2\Big).
\end{split}
\end{equation}

Combining \eqref{limitgradhsquared} and \eqref{limitpoisson}, we see that the index form from \eqref{indexformhardestsgrad} will satisfy
\begin{equation}\label{indexlimit1}
 \lim_{T\to\infty} \frac{4}{\pi T} I(Y,Y) =
\int_M \left\lvert \frac{d}{dr}\Big(\big(\alpha - u(r)\big)\xi(r)\Big) \right\rvert^2  \lvert \grad r\rvert^2 +
\left\lvert \big(\alpha - u(r)\big)  \xi(r)\right\rvert^2 \, \lvert \grad\theta\rvert^2
+ \frac{1}{\varphi(r)} \frac{d}{dr} \Big( \big(\alpha-u(r)\big) \xi(r)^2\Big) \, d\mu.
\end{equation}

Observing that the gradients of the coordinate functions in a general Riemannian metric are given by
$$ \grad r = \frac{g_{22} \partial_r - g_{12} \partial_{\theta}}{g_{11}g_{22}-g_{12}^2} \qquad \text{and}\qquad \grad \theta = \frac{-g_{12} \partial_r + g_{22} \partial_{\theta}}{g_{11}g_{22}-g_{12}^2},$$
formula \eqref{componentsdeterminant} shows that
$$ \lvert \grad r\rvert^2 = \frac{g_{22}(r,\theta)}{\varphi(r)^2} \qquad \text{and}\qquad \lvert \grad \theta\rvert^2 = \frac{g_{11}(r,\theta)}{\varphi(r)^2}.$$
Inserting these formulas into \eqref{indexlimit1}, and using the formula \eqref{areaformpolar} for the area element $d\mu$ gives
\begin{multline*}  \lim_{T\to\infty} \frac{4}{\pi T} I(Y,Y) =
\int_0^R \int_0^{2\pi} \frac{1}{\varphi(r)} \left\lvert \frac{d}{dr}\Big(\big(\alpha - u(r)\big)\xi(r)\Big) \right\rvert^2 g_{22}(r,\theta) +
\frac{1}{\varphi(r)} \left\lvert \big(\alpha - u(r)\big)  \xi(r)\right\rvert^2 \, g_{11}(r,\theta) \\
+ \omega(r) \frac{d}{dr} \Big( \big(\alpha-u(r)\big) \xi(r)^2\Big) \, d\theta \, dr.
\end{multline*}

The only terms depending on $\theta$ are $g_{11}(r,\theta)$ and $g_{22}(r,\theta)$, and performing the $\theta$ integration and using
the definitions \eqref{EGdef} of $E(r)$ and $G(r)$, we obtain
$$
\lim_{T\to\infty} \frac{4I(Y,Y)}{\pi T} =
2\pi \int_0^R \frac{G(r)}{\varphi(r)} \left\lvert \frac{d}{dr}\Big(\big(\alpha - u(r)\big)\xi(r)\Big) \right\rvert^2 +
\frac{E(r)}{\varphi(r)} \left\lvert \big(\alpha - u(r)\big)  \xi(r)\right\rvert^2  + \omega(r) \frac{d}{dr} \Big( \big(\alpha-u(r)\big) \xi(r)^2\Big) \, dr.$$
Finally we integrate the last term by parts in $r$ and use the fact that $\xi(0)=\xi(R)=0$ to obtain
$$
\lim_{T\to\infty} \frac{2I(Y,Y)}{\pi^2T} =
\int_0^R \frac{G(r)}{\varphi(r)} \left\lvert \frac{d}{dr}\Big(\big(\alpha - u(r)\big)\xi(r)\Big) \right\rvert^2  +
\frac{E(r)}{\varphi(r)} \left\lvert \big(\alpha - u(r)\big)  \xi(r)\right\rvert^2
- \omega'(r)  \big(\alpha-u(r)\big) \xi(r)^2  \, dr.$$
This limit is the quantity $I$ appearing in \eqref{alphacondition}; if it is  negative for some $\alpha\ne 0$, then for sufficiently large $T$ we can make the actual index $I(Y,Y)$ negative on the interval $[0,T]$, and so there is a conjugate point occurring at some time $\tau<T$.
\end{proof}


\begin{proof}[Proof of Corollary \ref{indexthmQversion}]
The easiest way to proceed is to simply compute the difference between $I_1$ given by \eqref{alphacondition} and $I_2$ given by \eqref{indexVversion}, and show that it is zero. We get
$$I_1 - I_2 =
\int_0^R \frac{G(r)}{\varphi(r)} \Big( u'(r)^2 \xi(r)^2 - 2 u'(r)\big(\alpha - u(r)\big) \xi(r)\xi'(r)\Big) - \big( \alpha - u(r)\big) \xi(r)^2 \, \frac{d}{dr} \big( \omega(r)-2v(r)\big) \, dr.
$$
Now note that by formula \eqref{vorticityGdef} and \eqref{vdef}, we have $\omega(r) - 2v(r) = \frac{G(r)u'(r)}{\varphi(r)}$, so that we get
$$ I_1-I_2 = \int_0^R \Big[ -\frac{G(r)u'(r)}{\varphi(r)} \, \frac{d}{dr}\Big( (\alpha - u(r)) \xi(r)^2\big) - (\alpha - u(r))\xi(r)^2 \, \frac{d}{dr} \Big(\frac{G(r)u'(r)}{\varphi(r)}\Big)\Big] \, dr,
$$
which is obviously the integral of a total derivative, and the function vanishes at the endpoints since $\xi(0)=\xi(R)=0$.

Similarly to prove this is equivalent to \eqref{indexQversion}, we show that $I_2-I_3$ simplifies to zero.
\begin{align*}
I_2-I_3 &= \int_0^R 2\big( \alpha - u(r)\big)^2 Q(r) \xi(r) \xi'(r) + \big(\alpha - u(r)\big)^2 Q'(r) \xi(r)^2 - 2 Q(r) \big( \alpha - u(r)\big) u'(r) \xi(r)^2 \, dr \\
&= \int_0^R \frac{d}{dr} \Big( \big(\alpha - u(r)\big)^2 Q(r) \xi(r)^2\Big) \, dr,
\end{align*}
and again this vanishes since $\xi(0)=\xi(R)=0$.
\end{proof}

\section{Proofs of Corollaries \ref{killingcorollary}, \ref{uprimecorollary}, and \ref{uprimeorigincorollary}}\label{corollaryproofsection}


\begin{proof}[Proof of Corollary \ref{killingcorollary}]
We assume that $u(r)=1$. Then formula \eqref{vorticityGdef} implies that
the vorticity is given by $\omega(r) = G'(r)/\varphi(r)$, and thus its derivative is given by
\begin{equation}\label{vorticityprime}
\omega'(r) = \frac{d}{dr} \left( \frac{G'(r)}{\varphi(r)}\right).
\end{equation}


The condition \eqref{alphacondition} for existence of conjugate points then takes the form (using $\beta = \alpha-1$ for simplicity)
\begin{equation}\label{alphaconditionisochron}
I = \int_0^R \frac{\beta^2}{\varphi(r)} \big[ G(r)\xi'(r)^2 + E(r)\xi(r)^2\big] - \beta \frac{d}{dr} \left( \frac{G'(r)}{\varphi(r)}\right) \xi(r)^2 \, dr < 0.
\end{equation}
For any fixed $\xi$, this takes the form
$ I(\beta) = \beta^2 A - \beta B $
for some positive number $A$ and some number
$B=\int_0^R \frac{d}{dr} \big( \frac{G'}{\varphi}\big) \xi^2 \, dr$, which reaches its minimum at
$$ I\left( \frac{B}{2A}\right) = -\frac{B^2}{2A}. $$
Hence as long as we can find a function $\xi$ that makes $B$ nonzero, we can make the index $I$ negative and obtain a conjugate point by Theorem \ref{mainindexthm}. The only way $B$ must be zero for every $\xi$ is if $\frac{d}{dr}(G'(r)/\varphi(r))$ is identically zero, i.e., if $G'(r) = c \varphi(r)$ for some $c\in\mathbb{R}$.

In particular in the rotationally symmetric case we have $G(r) = \varphi(r)^2$, so that $G'(r)/\varphi(r) = 2\varphi'(r)$, and the only way this is constant is if $\varphi''(r)=0$. Since $\varphi''(r) = -\kappa(r)\varphi(r)$ in terms of the Gaussian curvature, and $\varphi(r)$ is nonzero on $(0,R)$, the only way $B$ would be identically zero in the rotationally symmetric case is if the curvature is identically zero.
\end{proof}

The only cases in which we do not get conjugate points are rotational flows on flat spaces: e.g., rotation of a disc or annulus with $\varphi(r)=r$, or on a flat torus with $\varphi(r)\equiv 1$. It is interesting that either positive curvature or negative curvature of the underlying manifold $M$ could both contribute to conjugate points on $\Diffmu(M)$.

%

\begin{proof}[Proof of Corollary \ref{uprimecorollary}]
We use Corollary \ref{indexthmQversion}, in particular the formula \eqref{indexVversion}.
Choose $\alpha = u(r_0)$ in formula \eqref{indexVversion}. For a small $\varepsilon>0$, we consider functions $\xi$ having support in $(r_0-\varepsilon, r_0+\varepsilon)$, which is possible since $r_0$ is neither $0$ nor $R$.

Since $u'(r_0)=0$ and $u''(r_0)\ne 0$,  we have
$$ u(r)-\alpha = \tfrac{1}{2} u''(r_0) (r-r_0)^2 + O(\varepsilon^3).$$
Hence the condition \eqref{indexVversion} becomes
\begin{multline*}
I := \int_{r_0-\varepsilon}^{r_0+\varepsilon} \frac{\big[\tfrac{1}{2} u''(r_0) (r-r_0)^2 + O(\varepsilon^3)\big]^2}{\varphi(r_0) + O(\varepsilon)} \Big( \big[G(r_0)+O(\varepsilon)\big]\xi'(r)^2 + \big[E(r_0)+O(\varepsilon)\big]\xi(r)^2\Big) \\
+ 2\big[ v'(r_0)+O(\varepsilon)\big] \big[ \tfrac{1}{2} u''(r_0)(r-r_0)^2 + O(\varepsilon^3)\big] \xi(r)^2 \, dr < 0.
\end{multline*}


Changing variables to $s = \frac{r-r_0}{\varepsilon}$ and writing $\xi(r) = \zeta((r-r_0)/\varepsilon)$, this condition becomes
$$ I = \varepsilon^3 \int_{-1}^1 \tfrac{1}{4} u''(r_0)^2 s^4 \varphi(r_0) \zeta'(s)^2 - u(r_0)K(r_0) \varphi(r_0)  u''(r_0) s^2 \zeta(s)^2 \, ds + O(\varepsilon^4) < 0.$$
Clearly since $\varphi(r_0)>0$, for sufficiently small $\varepsilon>0$, we can make $I$ negative as long as for some function $\zeta$ supported in $[-1,1]$, we have
\begin{equation}\label{newcondition}
\tfrac{1}{4} u''(r_0)^2 \int_{-1}^1 s^4 \zeta'(s)^2 \, ds - K(r_0) u(r_0) u''(r_0) \int_{-1}^1 s^2 \zeta(s)^2 \, ds < 0.
 \end{equation}

The inequality
$$ \int_0^1 s^4 \zeta'(s)^2 \,ds \ge \frac{9}{4} \int_0^1 s^2 \zeta(s)^2 \,ds \qquad \text{if $\zeta(1)=0$}$$
is a special case of a general weighted inequality appearing in \cite{whipschains}; this case is easily proved by computing
\begin{align*}
0 &\le \int_0^1 \big( s^2 \zeta'(s) + \tfrac{3}{2} s \zeta(s)\big)^2 \,ds \\
&= \int_0^1 s^4 \zeta'(s)^2 \, ds + 3 \int_0^1 s^3 \zeta(s)\zeta'(s) \, ds + \tfrac{9}{4} \int_0^1 s^2 \zeta(s)^2 \, ds \\
&= \int_0^1 s^4 \zeta'(s)^2 \, ds - \frac{9}{4} \int_0^1 s^2 \zeta(s)^2 \, ds,
\end{align*}
after integrating the middle term by parts. The constant is optimal, as seen by using the test function
$ \zeta(s) = (1-s^2)^2 \lvert s\rvert^{-3/2+\delta}$ for $\delta>0$, which gives
\begin{equation}\label{testfunctioninequality}
\frac{\int_{-1}^1 s^4 \zeta'(s)^2 \, ds}{\int_{-1}^1 s^2 \zeta(s)^2 \, ds} = \frac{9}{4} + \frac{4\delta}{3} + \frac{\delta^2}{3}.
\end{equation}

Thus choosing a smooth $\zeta$ which is close in $H^1$ to such a near-optimal test function, with the ratio in \eqref{testfunctioninequality} smaller than $\tfrac{9}{4}+\varepsilon$, we find that the condition \eqref{newcondition} is equivalent to
\begin{equation}\label{newcondition2}
\tfrac{1}{4} u''(r_0)^2 (\tfrac{9}{4}+\varepsilon) - K(r_0) u(r_0) u''(r_0) < 0.
\end{equation}
Thus condition \eqref{newcondition2} holds for some $\varepsilon>0$ if and only if the condition \eqref{newcondition} holds.
\end{proof}

\begin{proof}[Proof of Corollary \ref{uprimeorigincorollary}]
As in the proof of Corollary \ref{uprimecorollary}, we set $\alpha = u(0)$.
For a small $\varepsilon>0$, we consider odd functions $\xi$ of $r$ having support in $(-\varepsilon, \varepsilon)$.

Since $u'(0)=0$ and $u''(0)\ne 0$,  we have
$$ u(r)-\alpha = \tfrac{1}{2} u''(0) r^2 + O(\varepsilon^3).$$
We also recall that by Lemma \ref{polarcoordsgeneral}, the functions $E$, $\varphi$, and $G$ behave like those in standard Euclidean coordinates.
That is, $E(r) = E(0) + O(\varepsilon)$, $\varphi(r) = \varphi'(0)r + O(\varepsilon^2)$, and $G(r) = \tfrac{1}{2} G''(0) r^2 + O(\varepsilon^3)$,
where all the lowest-order coefficients are positive.
The formula \eqref{vdef} gives
$$ v'(r) = \frac{d}{dr}\left( \frac{G'(r)u(r)}{2\varphi(r)}\right).$$
Now $G'(r)$ and $\varphi(r)$ are both odd functions of $r$, while $u(r)$ is an even function of $r$; thus $v$ is an even function of $r$, and we have
$v'(r) = v''(0)r + O(\varepsilon^2)$.

The condition \eqref{indexVversion} becomes
\begin{multline}\label{indexVorigin}
I := \int_0^{\varepsilon} \frac{\big[\tfrac{1}{2} u''(0) r^2 + O(\varepsilon^3)\big]^2}{\varphi'(0)r + O(\varepsilon^2)} \Big( \big[\tfrac{1}{2}G''(0)r^2+O(\varepsilon^3)\big]\xi'(r)^2 + \big[E(0)+O(\varepsilon)\big]\xi(r)^2\Big) \\
+ 2\big[ rv''(0)+O(\varepsilon^2)\big] \big[ \tfrac{1}{2} u''(0)r^2 + O(\varepsilon^3)\big] \xi(r)^2 \, dr < 0.
\end{multline}
So we just need to compute $v''(0)$.
Using L'Hopital's rule and the fact that all odd derivatives of $G$ and $u$ are zero at $r=0$, while all even derivatives of $\varphi(r)$ are zero at $r=0$, we get
\begin{equation}\label{vprimeprimezero}
v''(0) = \frac{u(0)\varphi'(0)G^{iv}(0) - u(0) \varphi'''(0) G''(0) + 3 u''(0) \varphi'(0) G''(0)}{6\varphi'(0)^2}.
\end{equation}
Again in the rotationally symmetric case, we have $G(r) = \varphi(r)^2$, so that we get the simplification
\begin{equation}\label{vpp0}
v''(0) = u(0) \varphi'''(0) + u''(0) \varphi'(0) = \big( -\kappa(0) u(0) + u''(0)\big) \varphi'(0).
\end{equation}

As in the proof of Corollary \ref{uprimecorollary}, we change variables to $s = \frac{r}{\varepsilon}$ and write $\xi(r) = \zeta(r/\varepsilon)$. Then the index \eqref{indexVorigin} becomes
$$ I = \varepsilon^3 \int_0^1  \bigg[ \frac{u''(0)^2G''(0)}{8\varphi'(0)} s^5\zeta'(s)^2 + \Big( \frac{u''(0)^2E(0)}{4\varphi'(0)}+ v''(0)u''(0)\Big)  s^3 \zeta(s)^2 \bigg]\, ds + O(\varepsilon^4) < 0.$$
For sufficiently small $\varepsilon>0$, we can make $I$ negative as long as for some function $\zeta$ supported in $[0,1]$, extending to a differentiable odd function on $[-1,1]$ we have
\begin{equation}\label{newconditionorigin}
\frac{G''(0)u''(0)^2}{8\varphi'(0)} \int_0^1 s^5 \zeta'(s)^2 \, ds + \Big(\frac{u''(0)^2E(0)}{4\varphi'(0)}+ v''(0)u''(0)\Big) \int_0^1 s^3 \zeta(s)^2 \, ds < 0.
\end{equation}

Using the same strategy as above, we prove that
$$ \int_0^1 s^5 \zeta'(s)^2 \,ds \ge 4 \int_0^1 s^3 \zeta(s)^2 \,ds \qquad \text{if $\zeta(1)=0$}$$
by computing
$$
0 \le \int_0^1 \big( s^3 \zeta'(s) + 2 s \zeta(s)\big)^2 \,ds,$$
again integrating the middle term by parts. Using the test function
$ \zeta(s) = s^{-2+\delta} (1-s)$ for $\delta>0$ gives
\begin{equation}\label{testfunctionorigininequality}
\frac{\int_0^1 s^5 \zeta'(s)^2 \, ds}{\int_0^1 s^3 \zeta(s)^2 \, ds} = \delta^2+\delta+4,
\end{equation}
so again we can get as close as we want by approximating this function.

Thus as in the proof of Corollary \ref{uprimecorollary}, the sufficient condition is
\begin{equation}\label{newconditionorigin2}
\frac{G''(0)u''(0)^2}{2\varphi'(0)} + \frac{u''(0)^2E(0)}{4\varphi'(0)}+ v''(0)u''(0) < 0.
\end{equation}
Using formula \eqref{vpp0}, we see that condition \eqref{newconditionorigin2} holds for some $\varepsilon>0$ if and only if the condition \eqref{originextremecondition} holds.

In the rotationally symmetric case we have $G(r) = \varphi(r)^2$, along with $E(0)=1$ and $\varphi'(0)=1$.
Using $\phi(r) = \phi'(0) r + \frac{1}{6} \phi'''(0) r^3 + \cdots$, we get
$$ G(r) = \phi'(0)^2 r^2 + \tfrac{1}{3} \phi'(0) \phi'''(0) r^4 + \cdots,$$
so that $G''(0) = 2\phi'(0)^2$ and $G^{iv}(0) = 8 \phi'(0) \phi'''(0) = -8 \kappa(0) \phi'(0)^2$. Plugging these into \eqref{originextremecondition}, we get the simplification \eqref{originextremerotational}.
\end{proof}

\section{Proof of Theorem \ref{kolmocase}}\label{kolmosection}

In this section we study the Kolmogorov flows, with stream function on the torus of the form $f(x,y) = -\cos{mx} \cos{ny}$ with $m$ and $n$ both natural numbers. We compute the flow explicitly in terms of Jacobi elliptic functions and find explicit formulas for all the terms appearing in the formula \eqref{indexQversion}. It is interesting that the index form is essentially the same for all of them, regardless of $m$ and $n$.
This becomes more obvious if we change the coordinates.

\begin{lemma}\label{noneuclideankolmo}
Suppose $m$ and $n$ are positive integers, and $f\colon \mathbb{T}^2\to\mathbb{R}$ is given by $f(x,y) = -\cos{mx}\cos{ny}$.
Define coordinates $(X,Y)$ by $X = \sin{mx}$ and $Y = \sin{ny}$. Then in these coordinates the metric becomes
\begin{equation}\label{metricuv}
ds^2 = \frac{dX^2}{m^2(1-X^2)} + \frac{dY^2}{n^2(1-Y^2)},
\end{equation}
with area form given by
\begin{equation}\label{areaformuv}
\mu = \frac{dX\wedge dY}{mn\sqrt{(1-X^2)(1-Y^2)}},
\end{equation}
and the function $f$ is given by
\begin{equation}\label{functionuv}
f(u,v) = -\sqrt{(1-X^2)(1-Y^2)}.
\end{equation}
The cell $C$ in Lemma \ref{polarcoordsgeneral} is given in $(X,Y)$ coordinates by the square $-1<X<1$ and $-1<Y<1$.
\end{lemma}

\begin{proof}
The critical points of $f$ happen when both $\cos{mx}\sin{ny}$ and $\sin{mx}\cos{ny}$ are zero, and these conditions imply that either $x=\frac{j\pi}{m}$ and $y=\frac{k\pi}{n}$, or $x=\frac{(j+1/2)\pi}{m}$ and $y=\frac{(k+1/2)\pi}{n}$, for integers $j$ and $k$. The origin is the former type and has $f_{xx}(0,0) = m^2$, $f_{xy}(0,0) = 0$, and $f_{yy}(0,0)=n^2$. Points of the latter type form the corners of a rectangular grid in the level set $f=0$, and the component of the set containing $(0,0)$ and having $f<0$ is given by $\lvert x\rvert < \frac{\pi}{2m}$ and $\lvert y\rvert< \frac{\pi}{2n}$. This becomes the square $\lvert X\rvert < 1$ and $\lvert Y\rvert < 1$ under the coordinate change, which is bijective on $C$.

Everything else is easy: since $x = \frac{1}{m} \arcsin{X}$ and $y=\frac{1}{n} \arcsin{Y}$, we get that the metric
$ ds^2 = dx^2 + dy^2$ becomes \eqref{metricuv}.
Thus the area form is \eqref{functionuv} by the standard definition of area form on a manifold (or by computing the Jacobian directly).
Clearly the function $f$ is given by \eqref{functionuv} on the cell $C$, where $f$ must be negative.
\end{proof}

In the $(X,Y)$ coordinates, the velocity field $U=\sgrad f$ is given by
\begin{equation}\label{Uinuvcoords}
U = \frac{1}{\varphi(X,Y)} \left( -\frac{\partial f}{\partial Y}\, \frac{\partial}{\partial X} + \frac{\partial f}{\partial X} \, \frac{\partial}{\partial Y}\right) = -mnY(1-X^2) \, \partial_X + mnX(1-Y^2) \, \partial_Y.
\end{equation}
To construct the polar coordinates as in Lemma \ref{polarcoordsgeneral}, we need to find the flow, which requires the use of Jacobi elliptic functions.

\begin{lemma}\label{polarkolmoprop}
For the velocity field $U$ given by \eqref{Uinuvcoords}, the polar coordinates defined by Lemma \ref{polarcoordsgeneral} can be chosen as
\begin{equation}\label{uvpolartau}
X = \frac{r \cn(\tau, r)}{\dn(\tau,r)}, \qquad Y = r \sn(\tau, r), \qquad \text{where } \tau = \frac{2\ellipticK(r) \theta}{\pi}.
\end{equation}
Here $\sn$, $\cn$, and $\dn$ are the usual Jacobi elliptic functions, while $\ellipticK(r)$ is the complete elliptic integral
$$ \ellipticK(r) = \int_0^{\pi/2} \frac{d\phi}{\sqrt{1-r^2\sin^2{\phi}}}.$$
\end{lemma}

\begin{proof}
We choose the curve $c$ in Lemma \ref{polarcoordsgeneral} to be $c(s) = (s,0)$ in the $(X,Y)$ coordinates. Then the flow $\gamma$ is obtained by solving the system
\begin{equation}\label{uvvectorfieldeqn}
\frac{dX}{dt} = -mnY(1-X^2), \qquad \frac{dY}{dt} = mnX(1-Y^2), \qquad X(0) = s, \qquad Y(0) = 0.
\end{equation}
Using the basic derivative formulas for the Jacobi elliptic functions (see \cite{elliptic} or \cite{gradry} for these and other formulas)
$$ \frac{d}{d\tau} \sn(\tau,r) = \cn(\tau,r) \dn(\tau,r), \qquad \frac{d}{d\tau} \cn(\tau,r) = -\sn(\tau,r)\dn(\tau,r), \qquad \frac{d}{d\tau} \dn(\tau,r) = -r^2 \sn(\tau,r)\cn(\tau,r)$$
together with the initial values
$$ \sn(0,r) = 0, \qquad \cn(0,r) = 1, \qquad \dn(0,r) = 1$$
and the basic identities
\begin{equation}\label{basicellipticidentities}
\cn(\tau,r)^2 + \sn(\tau,r)^2 = 1, \qquad r^2 \sn(\tau,r)^2 + \dn(\tau,r)^2 = 1, \qquad \dn(\tau,r)^2 - r^2 \cn(\tau,r)^2 = 1-r^2,
\end{equation}
we can compute that
\begin{equation}\label{flowuvcoords}
X(t,s) = \frac{s\cn(mn t,s)}{\dn(mnt,s)}, \qquad Y(t,s) = s \sn(mnt, s)
\end{equation}
satisfy the system \eqref{uvvectorfieldeqn}.

We have the following identities (analogous to the $(\theta\mapsto \pi/2-\theta)$ shift identities for standard trigonometric functions):
\begin{equation}\label{pihalfshift}
\sn(\tau+\ellipticK(r), r) = \frac{\cn(\tau,r)}{\dn(\tau,r)} \qquad \text{and}\qquad \frac{\cn(\tau+\ellipticK(r),r)}{\dn(\tau+\ellipticK(r),r)} = -\sn(\tau,r),
\end{equation}
which imply that the period of the flow \eqref{flowuvcoords} is $P(r) = \frac{4\ellipticK(r)}{mn}$. Thus the polar coordinate transformation is \eqref{uvpolartau}.
\end{proof}

Now we compute the velocity field, metric components, and area form in these polar coordinates.

\begin{lemma}\label{kolmometricstuff}
In polar coordinates defined by Lemma \ref{polarkolmoprop}, the velocity field is given by
\begin{equation}\label{Ukolmopolar}
U = \frac{\pi mn}{2\ellipticK(r)} \, \frac{\partial}{\partial \theta}
\end{equation}
while the area form is given by
\begin{equation}\label{areakolmopolar}
\mu = \varphi(r) \, dr\wedge d\theta, \qquad \text{where } \varphi(r) = \frac{2r\ellipticK(r)}{\pi mn\sqrt{1-r^2}}.
\end{equation}
The averages of the metric components from \eqref{EGdef} are given by
\begin{equation}\label{Gkolmo}
G(r) = \frac{4(m^2+n^2) r(1-r^2) \ellipticK(r)\ellipticK'(r)}{m^2n^2 \pi^2}
\end{equation}
and
\begin{equation}\label{Ekolmo}
E(r) = 
\frac{(m^2+n^2) J(r)}{m^2n^2(1-r^2)^2\ellipticK(r)},
\end{equation}
where $J$ is defined by 
\begin{equation}\label{Jdef}
J(r) = \int_0^{\ellipticK(r)}
 \big[ \sn(\tau,r)\dn(\tau,r) - \cn(\tau,r) Z(\tau,r)\big] \, d\tau.
\end{equation}
\end{lemma}

\begin{proof}
We first recall that the function is given in $(X,Y)$ coordinates by \eqref{functionuv}, and since our radial curve is $X=r$, $Y=0$, we have
$f(X,Y) = f(r,0)$ for all $(X,Y)$, so that
\begin{equation}\label{rintermsofuv}
(1-X^2)(1-Y^2) = 1-r^2 \qquad \text{or} \qquad r^2 = X^2 + Y^2 - X^2Y^2.
\end{equation}
Hence in particular we have
\begin{equation}\label{FintermsofRkolmo}
F(r) = -\sqrt{1-r^2} \text{ for } 0\le r\le 1.
\end{equation}
Formula \eqref{Ukolmopolar} follows from formula \eqref{Upolargeneral}, since the period is $P(r) = \frac{4\ellipticK(r)}{mn}$. Finally formula \eqref{velocitypolar} gives a quick way to compute the area element $\varphi(r)$, which is
$$ \varphi(r) = \frac{F'(r)}{u(r)} = \frac{2r\ellipticK(r)}{\pi mn\sqrt{1-r^2}}.$$
Of course we could also compute these quantities directly using the coordinate transformation \eqref{uvpolartau} and the usual transformation formulas.

Next we need to compute the metric averages $E(r)$ and $G(r)$ from \eqref{EGdef}. The easiest way to compute $G(r)$ is to use formula \eqref{vorticityGdef}. Since
$$\omega(r)\varphi(r) = -(m^2+n^2) F(r)\varphi(r),$$
formulas \eqref{FintermsofRkolmo}, \eqref{Ukolmopolar}, and \eqref{areakolmopolar} imply that
$$ \frac{\pi mn}{2} \frac{d}{dr}\Big( \frac{G(r)}{\ellipticK(r)}\Big) =
(m^2+n^2) \frac{ 2r \ellipticK(r)}{\pi mn},$$
and integrating gives
$$ G(r) = \frac{4(m^2+n^2)}{\pi^2 m^2 n^2} \, \ellipticK(r) \big[\ellipticE(r)-(1-r^2) \ellipticK(r)\big],$$
which is equivalent to \eqref{Gkolmo}.

Computing $E(r)$ is much more difficult, but it becomes somewhat easier if we first observe that by the shift formulas \eqref{pihalfshift}, the oddness of $\sn$ and evenness of $\cn$ and $\dn$, and the definitions \eqref{uvpolartau} of $u$ and $v$, we have the complementary formula
\begin{equation}\label{uvcomplementary}
X(r,\tfrac{\pi}{2}-\theta) = Y(r,\theta).
\end{equation}
Also using \eqref{pihalfshift} we can see that integrating any function involving only even powers of $X$, $X_{\theta}$, or $X_r$ around the full $[0,2\pi]$ interval is the same as integrating it from $0$ to $\pi/2$ and multiplying by $4$. Finally integrating any even power of $X$ or its derivatives from $0$ to $\pi/2$ is the same as integrating the corresponding power of a derivative of $Y$ backwards from $\pi/2$ to $0$, which gives the same answer by the obvious change of variables. Hence we can reduce everything to integrals of $Y(r,\theta)$.

We first compute
$$
g_{11}(r,\theta) = \frac{1}{m^2(1-X(r,\theta)^2)} \, \Big( \frac{\partial X}{\partial r}(r,\theta)\Big)^2 +
\frac{1}{n^2(1-Y(r,\theta)^2)} \, \Big( \frac{\partial Y}{\partial r(r,\theta)}\Big)^2.$$
Then using the symmetry described above, we see that
\begin{equation}\label{Estep1}
\begin{split}
E(r) &= \frac{1}{2\pi} \int_0^{2\pi} g_{11}(r,\theta)\,d\theta \\
&= \frac{2}{\pi} \left(\frac{1}{m^2} + \frac{1}{n^2}\right) \int_0^{\pi/2} \frac{1}{(1-Y(r,\theta)^2} \, \Big( \frac{\partial Y}{\partial r}(r,\theta)^2\Big) \, d\theta.
\end{split}
\end{equation}

Computing the radial derivative of $Y(r,\theta)$, we get (using formula 710.51 of \cite{elliptic} or (3.10.9) of \cite{lawden})
\begin{align*}
\frac{1}{\sqrt{1-Y(r,\theta)^2}} \frac{\partial Y(r,\theta)}{\partial r} &= \frac{1}{(1-r^2)\dn(\tau,r)} \, \Big( \sn(\tau,r)\dn^2(\tau,r)^2 - \dn(\tau,r) \cn(\tau,r) \zn(\tau,r)\Big) \\
&=  \frac{\sn(\tau,r)\dn^2(\tau,r)^2 - \dn(\tau,r) \cn(\tau,r) \zn(\tau,r)}{1-r^2},
\end{align*}
where $\zn$ is the Jacobi Zeta function defined by
$$ \frac{\partial}{\partial \tau} \zn(\tau,r) = \dn(\tau,r)^2 - \frac{\ellipticE(r) \tau}{\ellipticK(r)},$$
so that (by the known integral for $\dn^2$) we get $\zn(0,r)=0$ and $\zn(\ellipticK(r),r)=0$.

%
Thus we obtain from \eqref{Estep1} that
\begin{align*}
E(r) &= \frac{2(m^2+n^2)}{\pi m^2n^2(1-r^2)^2} \int_0^{\pi/2}
 \big[ \sn(\tau,r)\dn(\tau,r) - \cn(\tau,r) \zn(\tau,r)\big]^2 \, d\theta \\
 &= \frac{(m^2+n^2)}{m^2n^2(1-r^2)^2\ellipticK(r)} \int_0^{\ellipticK(r)}
 \big[ \sn(\tau,r)\dn(\tau,r) - \cn(\tau,r) \zn(\tau,r)\big]^2 \, d\tau,
 \end{align*}
 which is equation \eqref{Ekolmo}.
\end{proof}

Unfortunately we are not able to get a more explicit form of the function $J(r)$ from \eqref{Jdef}.
However this ends up not mattering, since we only need a lower bound for $E$, as in the next proof.

\begin{proof}[Proof of Theorem \ref{kolmocase}]
First we observe the dependence on the pair $(m,n)$ of positive integers. Let us denote by $\varphi_{mn}(r)$, $u_{mn}(r)$, $G_{mn}(r)$, and $E_{mn}(r)$ the formulas \eqref{areakolmopolar}, \eqref{Ukolmopolar}, \eqref{Gkolmo}, and \eqref{Ekolmo}, with their explicit dependence on $(m,n)$. Then we obviously have:
\begin{alignat*}{3}
\varphi_{mn}(r) &= \frac{1}{mn} \, \varphi_{11}(r), \qquad &\qquad  u_{mn}(r) &= mn u_{11}(r) \\
G_{mn}(r) &= \frac{m^2+n^2}{m^2n^2}\, G_{11}(r), & E_{mn}(r) &= \frac{m^2+n^2}{m^2n^2} \, E_{11}(r).
\end{alignat*}
Similarly if we compute $v_{mn}$ using \eqref{vdef}, we get
$$ v_{mn} = \frac{G_{mn}'(r)u_{mn}(r)}{2\varphi_{mn}(r)} =
(m^2+n^2) \, \frac{G_{11}'(r) u_{11}(r)}{2\varphi_{11}(r)} = (m^2+n^2) v_{11}(r),$$
so that using \eqref{Qdef} we get
$$ Q_{mn}(r) = \frac{m^2+n^2}{mn}\, Q_{11}(r).$$
We thus find that in the quantity $M_{mn}(r)$ defined by \eqref{Qconditioncurvature}, all the terms have exactly the same dependence on $(m,n)$, with
$$ M_{mn}(r) = \frac{(m^2+n^2)^2}{m^2n^2} \, M_{11}(r).$$
We conclude that when seeking variation fields supported in a single cell of the flow, the value of $(m,n)$ doesn't matter: every flow has energy-reducing variations in a cell if and only if the $(1,1)$ flow does.

Now we will show that the $(1,1)$ flow \emph{does not}. So from now on we set $m=n=1$ in all the formulas for $u$, $\varphi$, $E$, and $G$.  For this purpose it is easier, since we don't have an explicit formula for $E(r)$, to estimate from below using Cauchy-Schwarz: we have
\begin{align*}
E(r) G(r) &= \frac{1}{2\pi^2} \left( \int_0^{2\pi} g_{11}(r,\theta)\,d\theta\right) \left(\int_0^{2\pi} g_{22}(r,\theta) \,d\theta\right) \\
&\ge \frac{1}{2\pi^2} \left( \int_0^{2\pi} \sqrt{g_{11}(r,\theta)g_{22}(r,\theta)} \, d\theta\right)^2 \\
&\ge \frac{1}{2\pi^2} \left( \int_0^{2\pi} \varphi(r) \, d\theta\right)^2 \\
&= \varphi(r)^2.
\end{align*}
We conclude that
\begin{equation}\label{Mdef}
M_{11}(r) = \frac{G(r)Q'(r)}{\varphi(r)} + Q(r)^2 - \frac{E(r)G(r)}{\varphi(r)^2} \le \overline{M}(r) := \frac{G(r)Q'(r)}{\varphi(r)} + Q(r)^2 - 1.
\end{equation}

For the quantity $\overline{M}(r)$ defined by the right side, we have explicit formulas.
We compute $v_{11}(r)$ from \eqref{vdef}, using \eqref{Gkolmo}, \eqref{Ukolmopolar}, and \eqref{areakolmopolar}, as follows:
$$ v_{11}(r) = \frac{G'(r)u(r)}{2\varphi(r)} = \sqrt{1-r^2}
\left( 1 + \frac{(1-r^2)\ellipticK'(r)^2}{\ellipticK(r)^2}\right).$$
Then we compute from \eqref{Qdef} that
$$ Q_{11}(r) = \frac{v_{11}'(r)}{u_{11}'(r)} = \frac{2}{\pi r \sqrt{1-r^2}} \Big( \frac{r^2 \ellipticK(r)^2}{\ellipticK'(r)} - 2r(1-r^2) \ellipticK(r)
+ (2-3r^2)(1-r^2) \ellipticK'(r) + \frac{r(1-r^2)^2 \ellipticK'(r)^2}{\ellipticK(r)}\Big).$$
With these we can compute for example that for small $r$ we have
$$\overline{M}(r) = -\frac{1}{2} \,r^{2}-\frac{35}{64} \,r^{4} \mathrm{O}(r^{8}),$$
while for $r\approx 1$ we have
$$  \overline{M}(r) = -\frac{1}{\pi^2}
\Big( \frac{2}{1-r}+4 \big[ \ln{(1-r)} + \tfrac{15}{4}-3\ln{2}\big]^2+ O(1).$$ The graph plotted in Figure \ref{Mplot} shows that $\overline{M}(r)$ is always negative for $r\in (0,1)$, and thus so is $M(r)$. Now examining the index formula \eqref{indexQversion} for variations $\xi$ supported in $[0,R]$, we  see that it is positive-definite since $M(r)\le 0$, and thus no such variations can detect conjugate points.
\end{proof}

\begin{figure}[!ht]
\centering
\includegraphics[scale=0.5]{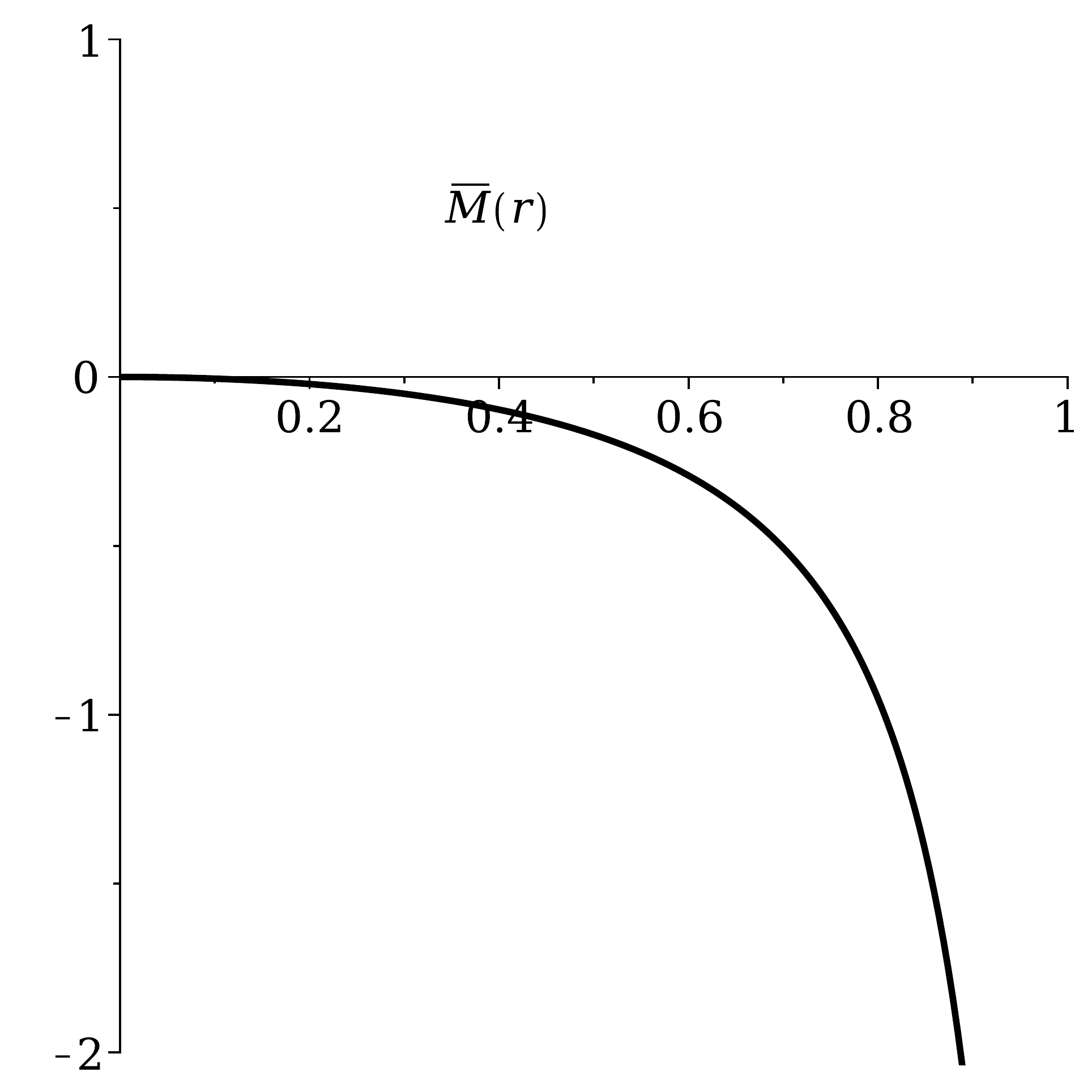}
\caption{A plot of the simplified upper bound $\overline{M}(r)$ from \eqref{Mdef}, for the quantity appearing in \eqref{indexQversion}. Negativity of this function everywhere implies that the index form is positive definite on variations supported in a single cell of any Kolmogorov flow on the torus.}\label{Mplot}
\end{figure}

Obviously there \emph{are} conjugate points along Kolmogorov flows, as found by Misio{\l}ek~\cite{misiolekconjugate} and Drivas et al.~\cite{drivasmisiolek}, so it must be that the variation fields \emph{cannot} be supported in a single cell. To use the computations in this paper to prove a general theorem on existence of conjugate points along Kolmogorov flows, one would need to patch together variation fields in different cells which are \emph{not} assumed to vanish on the boundary. The tricky bit here is that what seem like very well-behaved stream functions on the torus look rather bad in our polar coordinates, since the polar coordinates themselves degenerate on the cell boundaries. This difficulty seems quite surmountable, but would take us too far afield here, so we leave it for future work.

We thus find that for Kolmogorov flows on the flat torus, the method proposed in this paper does not improve upon the original M-criterion for detecting conjugate points. So instead we will apply that method in the Appendix to find a few more examples.

\section{Examples}\label{examplesection}

In this section we will work out some details for explicit examples where the criteria of Theorem \ref{mainindexthm} and its corollaries produce conjugate points.

First we give an example of Arnold-stable and nonisochronal flows with conjugate points, answering a question posed by Drivas et al.~\cite{drivasmisiolek}. Of course the isochronal flow on the same constant curvature hyperbolic disc will also be an Arnold-stable flow with conjugate points.

\begin{example}\label{hyperbolicdisc}
Consider the hyperbolic disc with metric $ds^2 = dr^2 + \sinh^2{r} \, d\theta^2$ for $r\in [0,R]$, for $R=\ln{\rho}$ with $\rho>1$.
Choose $F(r) = \frac{1}{2} \cosh^2{r}$, so that $u(r) = \cosh{r}$ and
$$\omega(r) = \frac{1}{\varphi(r)} \, \frac{d}{dr} \big( \phi(r)^2 u(r)\big) = 3\cosh^2{r} - 1 = 6F(r) - 1.$$
Then $\frac{\omega'(r)}{F'(r)} = 6$, which implies that $U$ satisfies the nonlinear Arnold stability criterion for steady 2D Euler flows.
The function $v$ from \eqref{vdef} is given by $v(r) = \cosh^2{r}$.

Choose test function $\xi(r) = \sinh{r} \sinh{(R-r)}$. Then the index given by \eqref{indexVversion} can be computed explicitly and gives
\begin{multline*}
I(\alpha,\rho) =
\frac{\left(\rho -1\right)^{4}}{3360\rho^5} \, \Big( 84 \rho^{2} \left(\rho^{2}+8 \rho +1\right) \alpha^{2}-105 \rho  (\rho +1)^4 \alpha \\
+ 23 \rho^6+92 \rho^5+181 \rho^4+248 \rho^3+181 \rho^2+92 \rho +23\Big).
\end{multline*}
This is minimized at
$$ \alpha = \frac{5(\rho +1)^4}{8 \rho(\rho^2+8 \rho +1)},$$
and for this $\alpha$ we obtain
$$ I = -\frac{(157 \rho^{4}+412 \rho^{3}+366 \rho^{2}+412 \rho +157) (\rho -1)^8}{53760 \rho^5 (\rho^{2}+8 \rho +1)},$$
which is obviously negative, so we get a conjugate point regardless of the radius $R=\ln{\rho}$.
\end{example}

Here we give an example of how to use the criterion from Corollary \ref{indexthmQversion} to select an appropriate test function $\xi$.

\begin{example}
Consider the standard $2$-sphere with $R=\pi$ and $\varphi(r) = \sin{r}$. Set $u(r) = \frac{7}{4}+4\cos{r} + \cos^2{r}$, which has $u'(r)<0$ for $r\in (0,\pi)$.
We compute that
$$ Q(r) := \frac{1}{u'(r)} \frac{d}{dr} \big( \phi'(r) u(r)\big) = \frac{7+32\cos{r}+12\cos^2{r}}{16+8\cos{r}},$$
which leads to $M$ defined by \eqref{Qconditioncurvature} being given by
$$ M(r) = \frac{3}{64} \, \frac{80\cos^4{r}+384\cos^3{r} + 496\cos^2{r}-
64\cos{r}-221}{(2+\cos{r})^2}.$$
This function $S$ is positive on $[0,0.3\pi]$ and again on $[0.83\pi,\pi]$ as shown in Figure \ref{example1fig}.
Choose
$$ \xi(r) = \begin{cases} \sin{4r} & 0\le r\le \frac{\pi}{4}, \\
0 & \text{otherwise.}\end{cases}$$
Then the index form \eqref{alphacondition} is given by
\begin{equation}\label{indexsphereexample}
I(\alpha) = \int_0^{\pi/4} \big[ u(r)-\alpha\big]^2 \Big[ \sin{r} \Big( \xi'(r) - \frac{Q(r)}{\varphi(r)} \xi(r)\Big)^2 -\frac{M(r)}{\sin{r}} \, \xi(r)^2\Big]  \, dr,
\end{equation}
and we can smooth $\xi$ out to a $C^{\infty}$ function at the expense of an arbitrarily small increase in $I$. As a function of $\alpha$, the index $I$ is given by
$$ I(\alpha) \approx 3.549
+\alpha^2 - 38.57\alpha + 104.2,$$
so that in particular $I(5.431) \approx -0.5635$; see the graph in Figure \ref{example1fig}. Thus $\eta(T)$ is eventually conjugate to the identity for sufficiently large $T$.
\end{example}

\begin{figure}[!ht]
\centering
\includegraphics[scale=0.4]{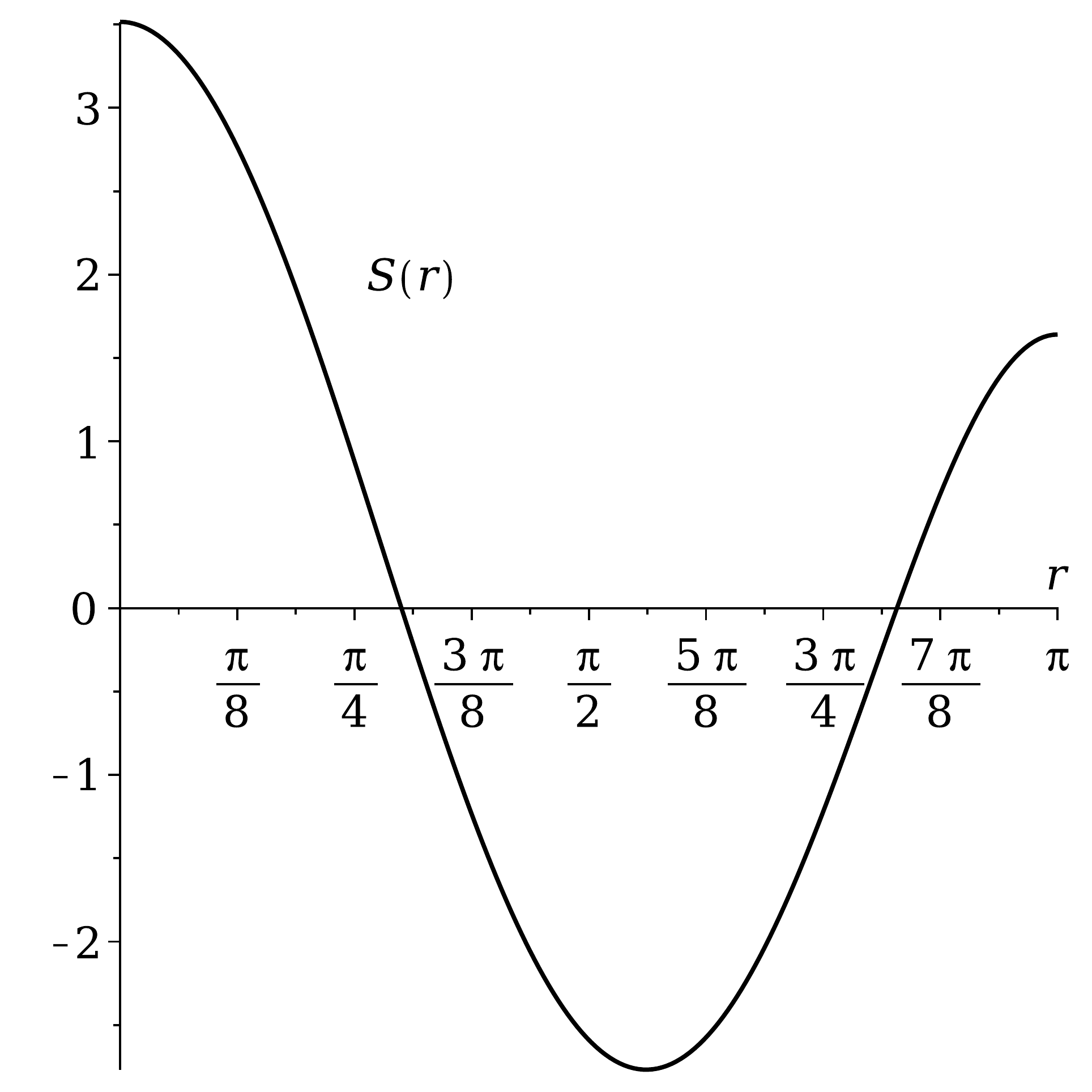} \qquad \includegraphics[scale=0.38]{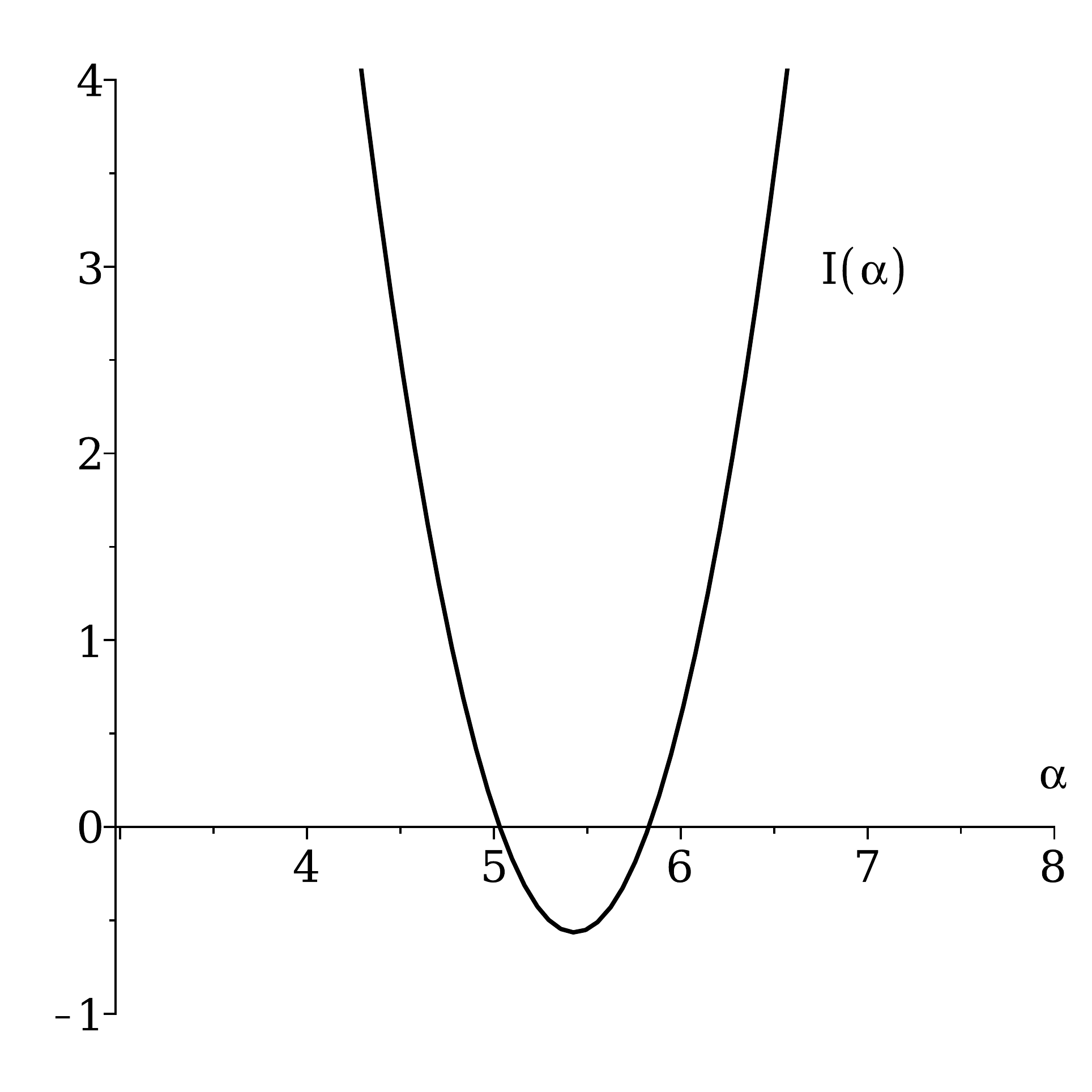}
\caption{On the left, the function $S = \varphi Q' + Q^2 - 1$, for $u$ defined on the round $2$-sphere by $u(r) = \frac{7}{4}+4\cos{r} + \cos^2{r}$. In particular it is positive on $[0,\frac{\pi}{4}]$, which is the support of the perturbation $\xi$.
On the right, the index form $I$ defined by \eqref{indexsphereexample}, considered for the function $\xi(r) = \sin{4r}$ on $[0,\tfrac{\pi}{4}]$ as a function of $\alpha$. }\label{example1fig}
\end{figure}

Next we give an example applying the local criterion of Corollary \ref{uprimecorollary}, in the case where $u$ has a local minimum in the interior.

\begin{example}
If $\varphi(r) = \sin{r}$ and $u(r) = \frac{9}{8} - \sqrt{2}\cos(r) + \cos^2(r)$, then $u'(\tfrac{\pi}{4})=0$ and
$\dfrac{u(\tfrac{\pi}{4})}{u''(\tfrac{\pi}{4})} = \frac{5}{8} > \frac{9}{16}$.
We compute that
$$\omega(r) = -3 \sqrt{2}\, \cos^2{r}+4 \cos^3{r}+\sqrt{2}+\tfrac{1}{4} \cos{r}.$$

If we choose $\epsilon=\delta=\frac{1}{10}$ in the notation of the proof of Corollary \ref{uprimecorollary}, and set
$$\xi(r) = \frac{\frac{\pi}{4}+\tfrac{1}{10}-r}{\lvert r-\frac{\pi}{4}\rvert^{\frac{7}{5}}} \qquad \text{for $r\in [\tfrac{\pi}{4}, \tfrac{\pi}{4}+\tfrac{1}{10}$,}$$
and zero otherwise, then  the integrand in \eqref{alphacondition} can be expressed in a series as
\begin{multline*}
\varphi(r) \left( \frac{d}{dr}\Big[ \big(\alpha - u(r))\xi(r)\Big]\right)^2 + \frac{1}{\varphi(r)} \big( \alpha - u(r)\big)^2 \xi(r)^2 - \omega'(r)\big( \alpha - u(r)\big) \xi(r)^2 = \\
-\frac{7 \sqrt{2}}{40000 (r-\frac{\pi}{4})^{\frac{4}{5}}}+\frac{51 \sqrt{2}\, \left(r-\frac{\pi}{4}\right)^{\frac{1}{5}}}{10000} + \cdots,
\end{multline*}
and this is integrable on $[\tfrac{\pi}{4},\tfrac{\pi}{4}+\frac{1}{10}]$, giving a negative value.

The variation field here is supported in a very small neighborhood of the circle $r=\tfrac{\pi}{4}$ on the sphere but approaches infinity as $r\to \tfrac{\pi}{4}$. Setting $\tilde{\xi}(r)$ to be the even extension of $\xi(r)$ around $r=\tfrac{\pi}{2}$, with $\tilde{\xi}(r) = \xi(\tfrac{\pi}{2}+\beta)$ for $\lvert r-\tfrac{\pi}{2}\rvert \le \beta$ for an even smaller $\beta$, will make this function finite and piecewise $C^1$, while still yielding a negative integral for \eqref{alphacondition}.
\end{example}

Finally we present an example of applying Corollary \ref{uprimeorigincorollary}, using also Lemma \ref{polarcoordsconverse} to demonstrate how this all can work in the non-rotationally symmetric case.

\begin{example}
Suppose $\varphi(r)=r$, $G(r) = r^2-\frac{r^4}{8}$, and $u(r) = 5+\tfrac{1}{2} r^2$. Let $\eta(r,\theta) = r^6 \cos{2\theta}$. By formulas \eqref{generalmetriccomponents}, we have
\begin{equation}\label{nonflatcomponents}
\begin{split}
g_{11} &= \frac{1152 r^8 \cos^2{2\theta} + 8 (r^2+10)^2}{(r^2+10)
(80 -r^4- 32r^2 \sin{2 \theta}-2 r^{2})} \\
g_{12} &= \frac{12r^5 \cos{2\theta}}{10+r^2} \\
g_{22} &= \frac{r^2(80 -r^4- 32r^2 \sin{2 \theta}-2 r^2)}{8(10+r^2)}.
\end{split}
\end{equation}
We easily verify that $g_{11}(0,\theta)=1$ so that $E(0)=1$, that
$ g_{11}g_{22}-g_{12}^2 = r^2$, and that the vorticity of $U$ is given by
$ \omega = 10r - \tfrac{1}{2} r^3 - \tfrac{3}{8} r^5$.

Since $E(0)=1$, $G''(0) = 2$, $G^{iv}(0) = -3$, $u(0)=5$, $u''(0)=1$, $\phi'(0)=1$, and $\phi'''(0)=0$, the condition \eqref{originextremecondition} becomes
$$ 3 + 24 < \frac{10}{1} \cdot \big(0 - (-3)\big),$$
which is of course satisfied. Hence there is a conjugate point along this geodesic, and the variation test field is supported in an arbitrarily small neighborhood of the origin.

To get a sense of the geometry of this very nonflat surface, we plot the sectional curvature on the disc of radius $R=1$, as computed in Maple; in particular it is positive near the origin but takes on both signs as we move away.
\begin{figure}[!ht]
\centering
\includegraphics[scale=0.5]{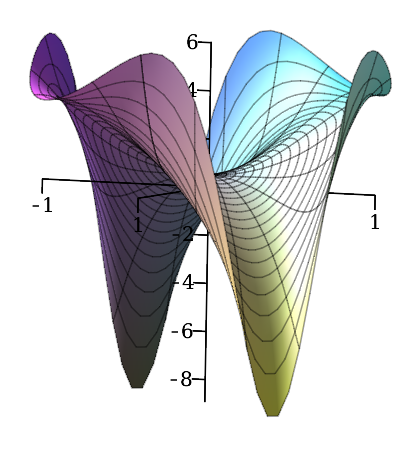}
\caption{The sectional curvature for the metric given by the components \eqref{nonflatcomponents} on the disc $r\le 1$. The steady flow velocity profile has a local minimum at the origin, and satisfies the condition of Corollary \ref{uprimeorigincorollary} to have an energy-reducing variation supported near the origin, but the surface itself does not have any rotational symmetry.}
\end{figure}

\end{example}

\appendix
\section{Appendix: The M-criterion for Kolmogorov flows on $\mathbb{T}^2$}\label{appendix}

As noted in Section \ref{kolmosection}, the criterion of Theorem \ref{mainindexthm} is not able to detect conjugate points along Kolmogorov flows on $\mathbb{T}^2$, where the stream functions are given by $f(x,y) = -\cos{mx}\cos{ny}$ for $m,n\in \mathbb{N}$. (Because of the symmetry between $x$ and $y$, there is no loss of generality in assuming $m\ge n$, which we will do throughout.)

In this case the Misio{\l}ek criterion becomes, using Corollary \ref{indexformcoro2D} with a stream function of the form $g(t,x,y) = \sin{\tfrac{t}{T}} \zeta(x,y)$, that
\begin{equation}\label{MisTorus}
\int_{\mathbb{T}^2} \lvert \nabla \phi\rvert^2 - (m^2+n^2) \phi^2 \, dx\,dy < 0, \qquad \text{where } \phi = \{f, \zeta\}
\end{equation}
for some function $\zeta\colon \mathbb{T}^2\to \mathbb{R}$.

The result of Drivas, Misio{\l}ek, Shi, and Yoneda~\cite{drivasmisiolek} can be stated as follows.
\begin{theorem}[DMSY]
If $(m,n)\in\mathbb{N}^2$ with $n\ge 2$ and $m>\frac{3n^2+6}{\sqrt{3}n}$, then the test function $\zeta(x,y) = \cos{(mx+y)}\cos{ny}$ gives a negative value in \eqref{MisTorus} on the stream function $f(x,y) = -\cos{mx}\cos{ny}$. Hence the corresponding geodesic eventually has a conjugate point.
\end{theorem}

In this appendix, we will fill out some of the diagram of Kolmogorov flows known to eventually have conjugate points, as shown in Figure \ref{kolmodiagram}.

\begin{figure}[!ht]
\centering
\includegraphics[scale=0.5]{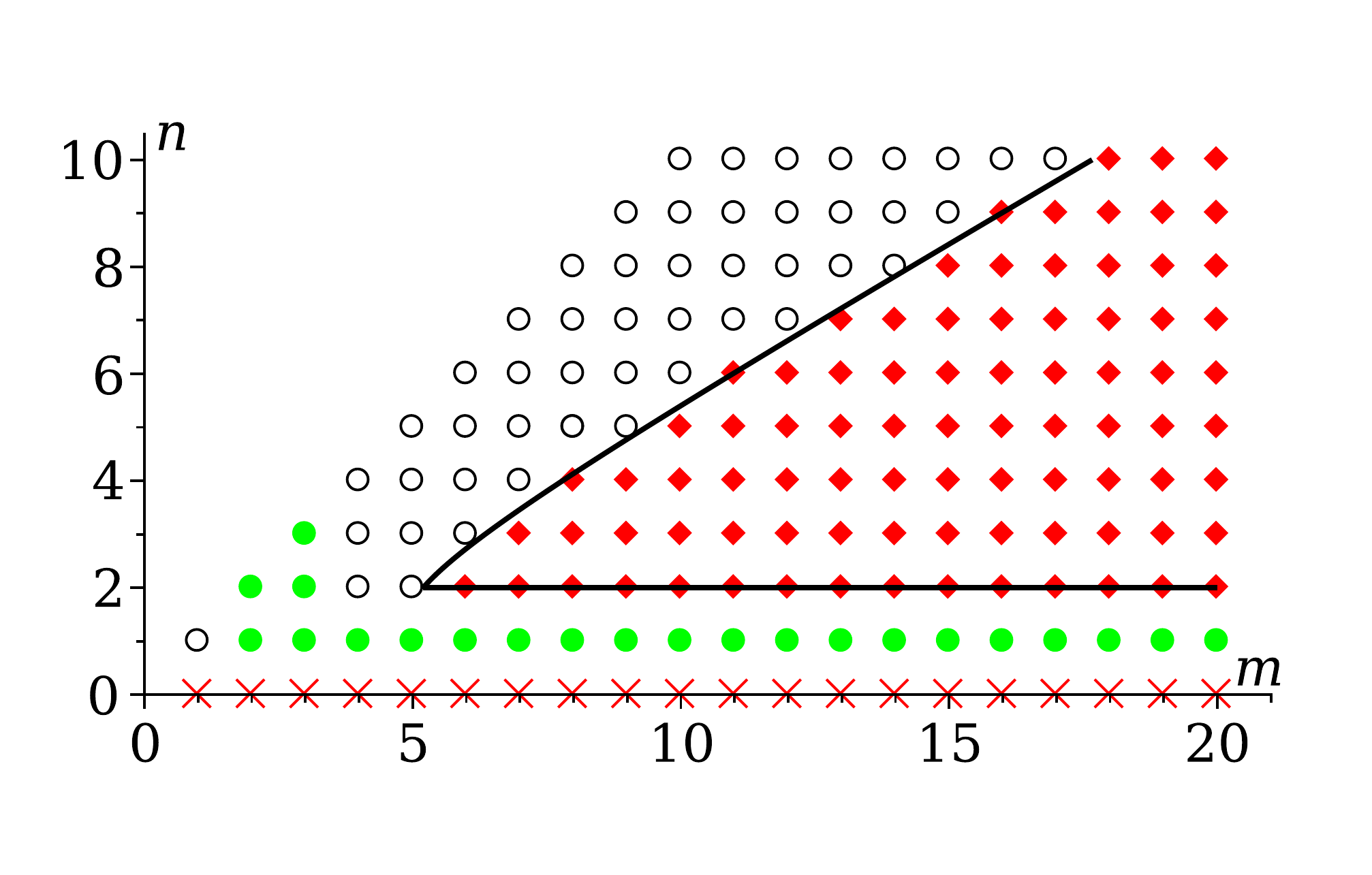}
\caption{Here we display the values of $(m,n)$ with $m\ge n$ for which the Kolmogorov flow $f(x,y) = -\cos{mx}\cos{ny}$ is known to have conjugate points. The diamond shapes in red,enclosed in the region $n\ge 2$,  $m>\frac{3n^2+6}{\sqrt{3}n}$, are those found by Drivas et al.~\cite{drivasmisiolek}. The red X marks when $n=0$ are known to have no conjugate points. The green circles are the new ones found here, while all unfilled circles are still unknown.}\label{kolmodiagram}
\end{figure}

\begin{theorem}\label{newkolmothm}
The geodesics corresponding to Kolmogorov flows $f(x,y) = -\cos{mx}\cos{ny}$ eventually have conjugate points if $n=1$ and $m\ge 2$.
In addition the pairs $(2,2)$, $(3,2)$, and $(3,3)$ also give geodesics with eventual conjugate points.
\end{theorem}

\begin{proof}
All we have to choose a test function $\zeta$ to make the M-criterion \eqref{MisTorus} satisfied in each case.

For $f(x,y) = -\cos{mx} \cos{y}$,
choose
$$ \zeta_{m1}(x,y) = \cos{x}(4+\cos{2y}).$$
Then
$$ \phi:= \{f,g\} =
2 m \sin{m x} \cos{y} \sin{2 y} \cos{x}+ \cos{m x}\sin{x} \sin{y} (4+\cos{2 y}),$$
and the index is
$$ \iota(\zeta,\zeta) = \iint_{\mathbb{T}^2}  \lvert \nabla \phi\rvert^2 - (m^2+1) \phi^2 \, dA = \frac{\pi^2(29-12m^2)}{4},$$
which is negative for all integers $m\ge 2$.

For $f(x,y) = -\cos{2x}\cos{2y}$, we choose
\begin{multline}
\zeta_{22}(x,y) = 235(\cos{x} + \cos{y}) - 27(\cos{3x} + \cos{3y})  -9(\cos{5x} + \cos{5y}) \\
 - 5(\cos{4x}\cos{5y} + \cos{5x}\cos{4y}) - 10(\cos{3x}\cos{4y} + \cos{4x}\cos{3y}).
\end{multline}
We then compute that if $\phi = \{f,\zeta\}$, then the index \eqref{MisTorus} is
$$ \iota(\zeta,\zeta) = \iint_{\mathbb{T}^2} \lvert \nabla \phi\rvert^2 - 8 \phi^2 \, dA = -854 \pi^2.$$

Finally for $f(x,y) = -\cos{3x}\cos{2y}$, choose
\begin{multline}
\zeta_{32}(x,y) = (1320\cos{6x} + 235\cos{12x})\cos{y} +
(1000 + 1060\cos{6x}+190\cos{12x} )\cos{3y}\\
+ (550 + 650\cos{6x} + 125\cos{12x})\cos{5y}
+ (195 + 270\cos{6x} + 63\cos{12x})\cos{7y} \\
+ (50 + 73\cos{6y} + 22\cos{12y})\cos{9y} +
(10 + 10\cos{6x} + 3\cos{12x})\cos{11y}.
\end{multline}
Then we obtain
$$ \iota(\zeta,\zeta) = \iint_{\mathbb{T}^2} \lvert \nabla \phi\rvert^2 - 13 \phi^2 \, dA = - \frac{764865\pi^2}{8}.$$

For $f(x,y) = -\cos{3x}\cos{3y}$, choose
$$
\zeta_{33}(x,y) = 1000\cos{x} - 18\cos{5x}\cos{6y} - 42\cos{5x} - 20\cos{7x} - 11\cos{7x}\cos{6y}.
$$
Then the index is
$$ \iota(\zeta,\zeta) = \iint_{\mathbb{T}^2} \lvert \nabla \phi\rvert^2 - 18 \phi^2 \, dA = -51939\pi^2.$$
\end{proof}

Once the test functions are guessed, verifying that the index is negative is straightforward. We can find them by searching in the space of Fourier coefficients, but especially as seen from the $(3,2)$ case, this starts to require many modes beyond those we start with. We have scaled the variation fields so that all the coefficients are integers (which allows us to compute the integrals exactly rather than numerically). The coefficients are quite sensitive: most of the ones in the $(3,2)$ case cannot be changed to a different integer without making the index positive. All computations here were done in Maple 2021. In Figure \ref{kolmofigure} we plot the variation fields whose formulas are given in the proof of Theorem \ref{newkolmothm}.

\begin{figure}[!ht]
\centering
\includegraphics[scale=0.4]{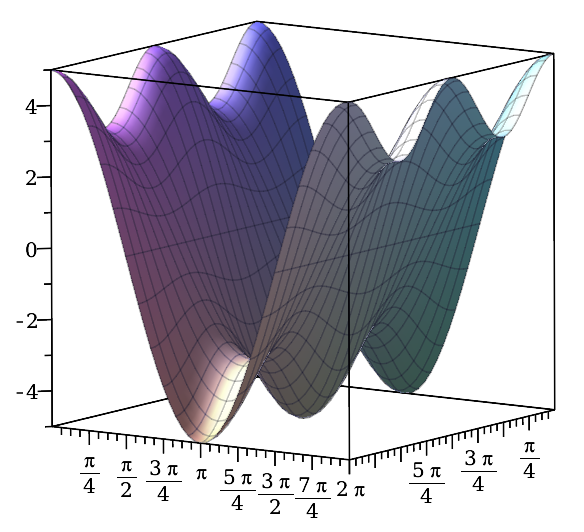} \qquad \includegraphics[scale=0.4]{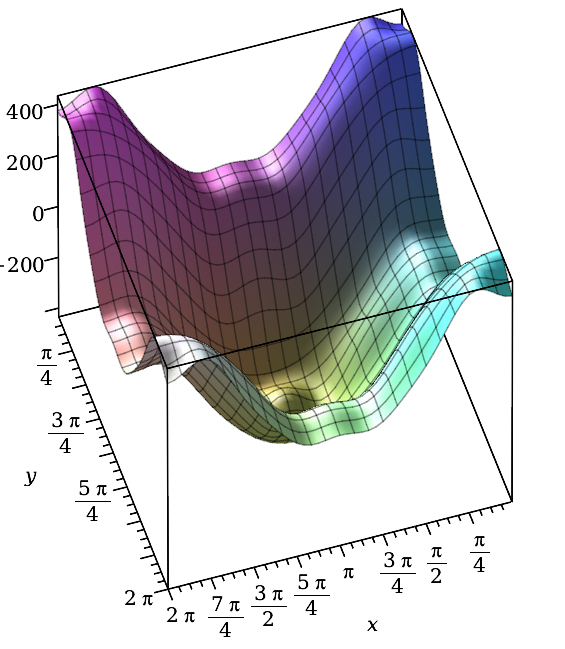} \\
\includegraphics[scale=0.4]{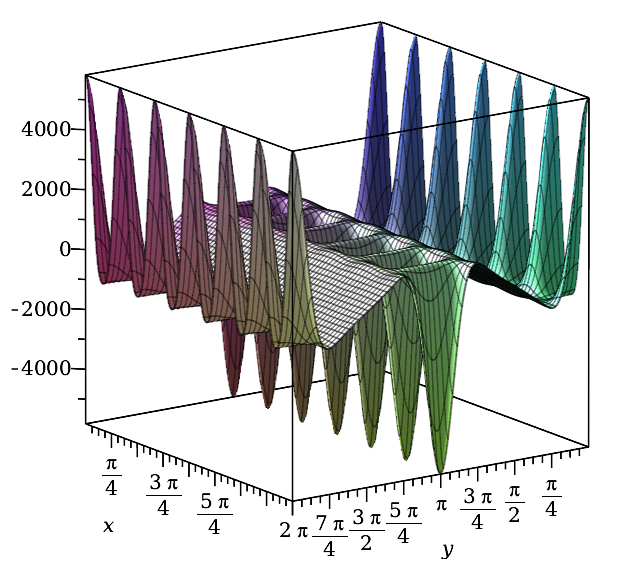} \qquad
\includegraphics[scale=0.4]{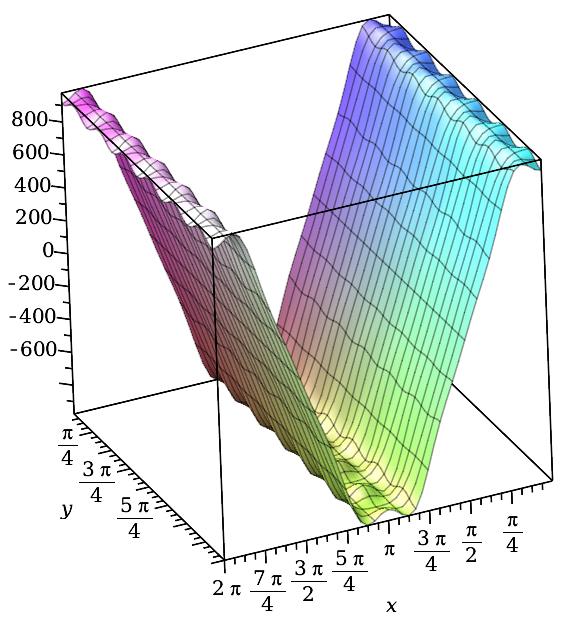} \\
\caption{Variation fields which eventually reduce the energy along specified Kolmogorov flows determined by parameters $(m,n)$. Top row: $(2,1)$ and $(2,2)$. Bottom row: $(3,2)$ and $(3,3)$.}\label{kolmofigure}
\end{figure}


\begin{thebibliography}{10}

\bibitem{Arn1966} V.I. Arnold, \emph{On the differential geometry of infinite-dimensional Lie groups and its application to the hydrodynamics of perfect fluids}, in Vladimir I. Arnold---collected works. Vol. II. Hydrodynamics, bifurcation theory, and algebraic geometry 1965--1972. Editors Givental and B.A. Khesin and A.N. Varchenko and V.A. Vassiliev and O.Ya. Viro. Springer-Verlag, Berlin, 2014.
\bibitem{AK1998} V. Arnold and B. Khesin, \emph{Topological nethods in hydrodynamics}, second edition, Springer-Verlag, New York, 2021.
\bibitem{benn} J. Benn, \emph{Conjugate points in $\Diffmu(S^2)$}, J. Geom. Phys. \textbf{170} 104369 (2021).
\bibitem{elliptic} P.F. Byrd and M.D. Friedman, \emph{Handbook of Elliptic Integrals for Engineers and Scientists}, Springer-Verlag New York, 1971.
\bibitem{docarmo} M.P. do Carmo, \emph{Riemannian Geometry}, Birkh{\"a}user, Boston 1992.
\bibitem{drivasmisiolek} T. Drivas, G. Misio{\l}ek, B. Shi, and T. Yoneda, \emph{Conjugate and cut points in ideal fluid motion}, Ann. Math. Qu{\' e}. \textbf{46} pp. 207--225 (2022).
\bibitem{EM1970} 	D. Ebin and J. Marsden,
	\textit{Groups of diffeomorphisms and the motion of an incompressible fluid},
	Ann. of Math. \textbf{92} (1970), 102-163.
\bibitem{EMP}	D. Ebin, G. Misio{\l}ek and S.C. Preston,
	\textit{Singularities of the exponential map on the volume-preserving diffeomorphism group},
	Geom. funct. anal. \textbf{16} (2006), 850-868.
\bibitem{gradry} I. S. Gradshteyn and I. M. Ryzhik, \emph{Table of Integrals, Series, and Products}, 8th edition, Academic Press, Cambridge, MA.
\bibitem{lawden} D.F. Lawden, \emph{Elliptic Functions and Applications}, Springer-Verlag, New York, 1989.
\bibitem{majdabertozzi} 	A. Majda and A. Bertozzi,
	\textit{Vorticity and Incompressible Flow},
	Cambridge University Press, Cambridge 2001.
\bibitem{Mis1993} 	G. Misio{\l}ek,
	\textit{Stability of ideal fluids and the geometry of the group of diffeomorphisms},
	Indiana Univ. Math. J. \textbf{42} (1993), 215--235.
\bibitem{misiolekconjugate} 	G. Misio{\l}ek,
	\textit{Conjugate points in $\mathcal{D}_\mu(\mathbb{T}^2)$},
	Proc. Amer. Math. Soc. \textbf{124} (1996), 977-982.
\bibitem{misiolekmorse} G. Misio{\l}ek, \emph{The exponential map near conjugate points in 2D hydrodynamics}, 
Arnold Math. J. \textbf{1} (2015), 243--251.
\bibitem{hardest} 	S.C. Preston,
	\emph{On the volumorphism group, the first conjugate point is always the hardest},
	Commun. Math. Phys. \textbf{267} (2006), 493-513.
\bibitem{nonpositive} S.C. Preston, \emph{Nonpositive curvature on the area-preserving diffeomorphism group}, J. Geom. Phys. \textbf{53} (2005).
\bibitem{prestonstability} S.C. Preston,
\emph{For ideal fluids, Eulerian and Lagrangian instabilities are equivalent},
{\it Geom. Func. Anal.} \textbf{14} no. 5, (2004), 1044--1062.
\bibitem{prestonthesis} S.C. Preston, \emph{Eulerian and Lagrangian stability of fluid motions}, Ph. D. Thesis, SUNY Stony Brook, 2002.
\bibitem{whipschains} S.C. Preston, \emph{The motion of whips and chains}, J. Differential Equations \textbf{3} (2011) 504--550.
\bibitem{wkb} S.C. Preston, \emph{The WKB method for conjugate points in the volumorphism group}, Indiana Univ. Math. J. \textbf{57} no. 7 (2008), 3303--3328.
\bibitem{tauchiyonedaellipsoid} T. Tauchi and T. Yoneda, \emph{Existence of a conjugate point in the incompressible Euler flow on an ellipsoid}, J. Math. Soc. Japan Advance Publication pp. 1--25 (2021).
\bibitem{tauchiyonedaarnold} T. Tauchi and T. Yoneda, \emph{Arnold stability and Misiołek curvature} https://arxiv.org/abs/2110.04680
\bibitem{tauchiyonedacoriolis} T. Tauchi and T. Yoneda, \emph{Positivity for the curvature of the diffeomorphism group corresponding to the incompressible Euler equation with Coriolis force}, Prog. Theor. Exp. Phys. 210043 (2021).
\bibitem{washabaugh} P. Washabaugh and S.C. Preston, \emph{The geometry of axisymmetric ideal fluid flows with swirl}, Arnold Math. J. \textbf{3}(2) (2017), 175--185.
\bibitem{yudovich} V.I. Yudovich, \emph{Non-stationary flow of an ideal incompressible liquid}, USSR Comp. Math. \& Math. Phys. \textbf{3}(6) (1963), 1407--1456.

\end{thebibliography}
\end{document}